\documentclass[a4paper, 12pt]{amsart}

\usepackage{amsmath,amsthm,amssymb}
\usepackage[margin=1in, includefoot, footskip=30pt]{geometry}
\usepackage{layout}
\usepackage{graphicx}
\usepackage{epstopdf}
\usepackage{subcaption}
\usepackage{float}
\usepackage{stmaryrd}
\usepackage{tikz}
\usetikzlibrary{patterns}
\usepackage{mathtools}
\usepackage[colorlinks=true, citecolor=blue, linkcolor=black]{hyperref}
\usepackage[capitalise,nameinlink,noabbrev]{cleveref}
\hypersetup{final}
\usepackage{enumerate}

\newtheorem{lemma}{Lemma}
\newtheorem*{lemma-non}{Lemma}
\newtheorem{proposition}{Proposition}
\newtheorem{theorem}{Theorem}
\newtheorem*{theorem-non}{Theorem}

\newtheorem{corollary}{Corollary}
\newtheorem*{corollary-non}{Corollary}
\newtheorem{remark}{Remark}

\def\R{\mathbb R}

\def\N{\mathbb N}

\def\p{\partial}
\DeclareMathOperator{\supp}{supp}
\DeclareMathOperator*{\esssup}{ess\,sup}

\newcommand{\norm}[1]{\left\|#1 \right\|}

\let\div\relax
\DeclareMathOperator{\div}{div}
\def\grad{\nabla}

\newcommand{\jump}[1]{\llbracket #1 \rrbracket }

\date{\today}

\title[Stabilized FEM for ill-posed convection--diffusion problems. I]{A stabilized finite element method for inverse problems subject to the convection--diffusion equation. I: diffusion-dominated regime}

\usepackage[foot]{amsaddr}

\author{Erik Burman}
\author{Mihai Nechita}
\author{Lauri Oksanen}

\address{Department of Mathematics, University College London, Gower Street, London UK, WC1E 6BT.}
\email{\{e.burman, mihai.nechita.16, l.oksanen\}@ucl.ac.uk}

\begin{document}

\begin{abstract}
The numerical approximation of an inverse problem subject to the convection--diffusion equation when diffusion dominates is studied.
We derive Carleman estimates that are of a form suitable for use in numerical analysis and 
with explicit dependence on the P\'eclet number. A stabilized finite element method is then proposed and analysed.
An upper bound on the condition number is first derived. Combining the stability estimates on the continuous problem
with the numerical stability of the method, we then obtain error estimates in local $H^1$- or $L^2$-norms that are optimal
with respect to the approximation order, the problem's stability and perturbations in data.
The convergence order is the same for both norms, but the $H^1$-estimate requires an additional divergence assumption for the convective field.
The theory is illustrated in some computational examples.
\end{abstract}	
	
\maketitle

\section{Introduction}\label{sec:intro}
We consider the convection--diffusion equation
\begin{equation}\label{eq:model_problem}
\mathcal{L} u := - \mu \Delta u + \beta \cdot \nabla  u = f \quad \text{in } \Omega,
\end{equation}
where $\Omega \subset \R^n$ is open, bounded and connected, $\mu>0$ is the diffusion coefficient and $\beta\in [W^{1,\infty}(\Omega)]^n$ is the convective velocity field. We assume that no information is given on the boundary $\partial \Omega$ and that there exists a solution $u\in H^2(\Omega)$ satisfying \eqref{eq:model_problem}. For an open and connected subset $\omega \subset \Omega$, define the perturbed restriction $\tilde U_\omega := u\vert_{\omega} + \delta$, where $\delta \in L^2(\omega)$ is an unknown function modelling measurement noise.
The data assimilation (or unique continuation) problem consists in finding $u$ given $f$ and $\tilde U_\omega$. Here the coefficients $\mu$ and $\beta$, and the source term $f$ are assumed to be known. This linear problem is ill-posed and it is closely related to the elliptic Cauchy problem, see e.g. \cite{ARRV09}. Potential applications include for example flow problems for which full boundary data are not accessible, but where local measurements (in a subset of the domain or on a part of the boundary) can be obtained. 

The aim is to design a finite element method for data assimilation with weakly consistent
regularization applied to the convection--diffusion equation \eqref{eq:model_problem}. In the present analysis we consider the regime where diffusion dominates and in the companion paper \cite{BNO19b} we treat the one with dominating convective transport. To make this more precise we introduce the P\'eclet number associated to a given length scale
$l$ by
$$
Pe(l) := \frac{|\beta| l}{\mu},
$$
for a suitable norm $|\cdot|$ for $\beta$.
If $h$ denotes the characteristic length scale of the computation, we
define the diffusive regime by
$Pe(h)<1$  and the convective regime by $Pe(h)>1$.
It is known that the character of the system changes drastically in the two regimes and we therefore need to apply different concepts of stability in the two cases.
In the present paper we assume that the P\'eclet number is small and we use an approach similar to that employed for the Laplace equation in \cite{Bur14b}, for the Helmholtz equation in \cite{BNO19} and for the heat equation in \cite{BO18}, that is we combine conditional stability estimates for the physical problem with optimal numerical stability obtained using a bespoke
weakly consistent stabilizing term.
For high P\'eclet numbers on the other hand, we prove in \cite{BNO19b} weighted estimates directly on the discrete solution, that reflect the anisotropic character of the
convection--diffusion problem.

In the case of optimal control problems subject to convection-diffusion problems that are well-posed, there are several works in the literature on stabilized finite element methods. In \cite{DQ05} the authors considered stabilization 
using a Galerkin least squares approach in the Lagrangian. Symmetric stabilization in the form of local projection stabilization was proposed in \cite{BV07} and using penalty on the gradient jumps
in \cite{YZ09,HYZ09}. The key difference between the well-posed case and the ill-posed case that we consider herein is that we can not use stability of neither the forward nor the backward equations. Crucial instead is the
convergence of the weakly consistent stabilizing terms and the matching of the quantities in the discrete method and the available (best) stability of the continuous problem. Such considerations lead to results
both in the case of high and low P\'eclet numbers, but the different stability properties in the two regimes lead to a different analysis for each case that will be considered in the two parts of this paper.

The main results of this current work are the convergence estimates with explicit dependence on the P\'eclet number in \cref{L^2-error} and \cref{H^1-error}, that rely on the continuous three-ball inequalities in \cref{lem:shifted3b} and \cref{cor:3_balls_impr}.   

\section{Stability estimates}\label{sec:stability_estimates}
We prove conditional stability estimates for the unique continuation problem subject to the convection--diffusion equation \eqref{eq:model_problem} in the form of three-ball inequalities, see e.g. \cite{MV12} and the references therein. The novelty here is that we keep track of explicit dependence on the diffusion coefficient $\mu$ and the convective vector field $\beta$. The first such inequality is proven in \cref{cor:3_balls}, followed by \cref{lem:shifted3b} and \cref{cor:3_balls_impr}, where the norms for measuring the size of the data are weakened to serve the purpose of devising a finite element method in \cref{sec:FEM}.

First we prove an auxiliary logarithmic convexity inequality, which is a more explicit version of \cite[Lemma 5.2]{LeRL12}.
\begin{lemma}\label{lem:log-convexity}
	Suppose that $a, b, c \ge 0$ and $p, q > 0$ satisfy
	$c \le b$ and 
	$c \le e^{p\lambda} a + e^{-q\lambda} b$
	for all $\lambda > \lambda_0 \ge 0$.
	Then there are $C > 0$ and $\kappa \in (0,1)$ (depending only on $p$ and $q$) 
	such that $$c \le C e^{q \lambda_0} a^\kappa b^{1-\kappa}.$$
\end{lemma}
\begin{proof}
	We may assume that $a, b > 0$, since $c = 0$ if $a = 0$ or $b = 0$.
	The minimizer $\lambda_*$ 
	of the function $f(\lambda) = e^{p\lambda} a + e^{-q\lambda} b$ is given by
	$$
	\lambda_* = \frac{1}{p+q} \log \frac{qb}{pa},
	$$
	and writing $r = q/p$, the minimum value is
	$$
	f(\lambda_*)
	= a\left(\frac{qb}{pa}\right)^{p/(p+q)}  
	+ b\left(\frac{qb}{pa}\right)^{-q/(p+q)} 
	= \left(r^{p/(p+q)} + r^{-q/(p+q)} \right)a^{q/(p+q)} b^{p/(p+q)}.
	$$
	This shows that if $\lambda_* > \lambda_0$ then
	$$
	c \le C_1 a^\kappa b^{1-\kappa},
	$$
	where $\kappa = q/(p+q)$ and $C_1 = r^{p/(p+q)} + r^{-q/(p+q)}$.
	On the other hand, if $\lambda_* \le \lambda_0$
	then it holds that 
	$e^{-q \lambda_0} \le e^{-q \lambda_*} = a^{q/(p+q)} (r b)^{-q/(p+q)}$, or equivalently,
	$$
	b^{q/(p+q)} \le e^{q \lambda_0} a^{q/(p+q)} r^{-q/(p+q)}.
	$$
	Therefore
	$$
	c \le b = b^{q/(p+q)} b^{p/(p+q)} \le e^{q \lambda_0} r^{-q/(p+q)} a^{q/(p+q)} b^{p/(p+q)}.
	$$
	That is, if $\lambda_* \le \lambda_0$ then
	$$
	c \le C_2 e^{q \lambda_0} a^\kappa b^{1-\kappa},
	$$
	where $C_2 = r^{-q/(p+q)}$.
	As $e^{q \lambda_0} \ge 1$ and $C_1 > C_2$, 
	the claim follows by taking $C = C_1$.
\end{proof}

The following Carleman inequality is well-known, see e.g. \cite{LeRL12}. For the convenience of the reader we have included an elementary proof in \cref{appendix}.
\begin{proposition}
	\label{prop:carleman}
	Let $\rho \in C^3(\Omega)$ and $K \subset \Omega$ be a compact set that does not contain critical points of $\rho$.
	Let $\alpha,\tau>0$ and $\phi = e^{\alpha \rho}$. Let $w \in C^2_0(K)$ and $v = e^{\tau \phi} w$. Then there is $C>0$ such that
	$$
	\int_K e^{2\tau\phi} (\tau^3 w^2 + \tau |\nabla w|^2) \,\mathrm{d}x   
	\le C \int_K e^{2 \tau\phi} |\Delta w|^2 \,\mathrm{d}x,
	$$
	for $\alpha$ large enough and $\tau \ge \tau_0$, where $\tau_0 > 1$ depends only on $\alpha$ and $\rho$.
\end{proposition}

Using the above Carleman estimate we prove a three-ball inequality that is explicit with respect to $\mu$ and $\beta$, i.e. the constants in the inequality are independent of the P\'eclet number. The corresponding inequality with constant depending implicitly on the P\'eclet number is proven for instance in \cite{MV12}. We denote by $B(x,r)$ the open ball of radius $r$ centred at $x$, and by $d(x,\partial \Omega)$ the distance from $x$ to the boundary of $\Omega$.
\begin{corollary}
	\label{cor:3_balls}
	Let $x_0 \in \Omega$ and $0 < r_1 < r_2 < d(x_0,\partial \Omega)$.
	Define $B_j = B(x_0, r_j)$, $j=1,2$.
	Then there are $C > 0$ and $\kappa \in (0,1)$ such that for $\mu>0$, $\beta\in [L^\infty(\Omega)]^n$ and $u \in H^2(\Omega)$ it holds that
	$$
	\norm{u}_{H^1(B_2)} \le C e^{C \tilde{Pe}^2} \left(\norm{u}_{H^1(B_1)} + \frac{1}{\mu} \norm{\mathcal{L} u}_{L^2(\Omega)}\right)^\kappa \norm{u}_{H^1(\Omega)}^{1-\kappa},
	$$
	where $\tilde{Pe} = 1+|\beta|/\mu$ and $|\beta| = \norm{\beta}_{[L^\infty(\Omega)]^n}$.
	\begin{proof}
		Due to the density of $C^2(\Omega)$ in $H^2(\Omega)$, it is enough to consider $u \in C^2(\Omega)$. Let now $0<r_0<r_1$ and $r_2 < r_3 < r_4 < d(x_0,\partial \Omega)$.
		We choose non-positive $\rho \in C^\infty(\Omega)$ 
		such that $\rho(x) = -d(x,x_0)$ outside $B_0$.
		Since $|\nabla \rho| = 1$ outside $B_0$, $\rho$ does not have critical points in $B_4 \setminus B_0$.
		Let $\chi \in C_0^\infty(B_4 \setminus B_0)$ satisfy $\chi = 1$ in $B_3 \setminus B_1 $, and set $w = \chi u$.
		We apply \cref{prop:carleman} with $K = \overline{B_4} \setminus B_0$ to get
		\begin{align}\label{eq:Carleman-Laplace}
		\mu^2 \int_{B_4 \setminus B_0} (\tau^3 |w|^2 + \tau |\nabla w|^2) e^{2\tau\phi}\, \,\mathrm{d}x
		\le C \int_{B_4 \setminus B_0} |\mu \Delta w|^2 e^{2 \tau\phi}\, \,\mathrm{d}x,
		\end{align}
		for $\phi = e^{\alpha \rho}$, with large enough $\alpha >0$, and $\tau \ge \tau_0$ (where $\tau_0 > 1$ depends only on $\alpha$ and $\rho$).
		We bound from above the right-hand side by a constant times
		$$
		\int_{B_4 \setminus B_0} |\mu \Delta w - \beta \cdot \grad w|^2 e^{2 \tau\phi}\, \,\mathrm{d}x  + |\beta|^2 \int_{B_4 \setminus B_0} |\grad w|^2 e^{2 \tau\phi}\, \,\mathrm{d}x.
		$$
		Taking $\tau \ge 2 |\beta|^2 / \mu^2$, the second term above is absorbed by the left-hand side of \eqref{eq:Carleman-Laplace} to give
		\begin{align}
		\label{balls_step2}
		\mu^2 \int_{B_4 \setminus B_0} (\tau^3 |w|^2 + \frac{\tau}{2} |\nabla w|^2) e^{2\tau\phi}\, \,\mathrm{d}x
		\le C \int_{B_4 \setminus B_0} |\mu \Delta w - \beta \cdot \grad w|^2 e^{2 \tau\phi}\, \,\mathrm{d}x.
		\end{align}
		Since $\phi\le 1$ everywhere, by defining $\Phi(r) = e^{-\alpha r}$ we now bound from below the left-hand side in \eqref{balls_step2} by
		$$
		\mu^2 \int_{B_2 \setminus B_1} (\tau^3 |w|^2 + \tau |\nabla w|^2) e^{2\tau\phi}\, \,\mathrm{d}x
		\ge
		\mu^2 \tau e^{2\tau \Phi(r_2)} \norm{u}^2_{H^1(B_2)} - \mu^2 \tau e^{2\tau} \norm{u}^2_{H^1(B_1)}.
		$$
		An upper bound for the right-hand side in \eqref{balls_step2} is given by
		\begin{align*}
		&C \int_{B_4} |\mu \Delta u - \beta \cdot \grad u|^2 e^{2 \tau\phi}\, \,\mathrm{d}x
		+ C \int_{(B_4 \setminus B_3) \cup B_1} |(\mu [\Delta, \chi] - \beta \cdot \grad \chi)u|^2 e^{2 \tau\phi}\, \,\mathrm{d}x
		\\
		&\quad \le C e^{2 \tau} \norm{\mu \Delta u - \beta \cdot \grad u}^2_{L^2(B_4)}
		+ C e^{2 \tau\Phi(r_3)} (\mu^2 + |\beta|^2) \norm{u}^2_{H^1(B_4 \setminus B_3)}
		+ C e^{2 \tau} (\mu^2 + |\beta|^2) \norm{u}^2_{H^1(B_1)}. 
		\end{align*}
		Combining the last  two inequalities we thus obtain that
		\begin{align*}
		\mu^2 e^{2\tau \Phi(r_2)} \norm{u}^2_{H^1(B_2)}
		&\le
		C e^{2\tau} \left( (\mu^2 + |\beta|^2) \norm{u}^2_{H^1(B_1)}
		+ \norm{\mu \Delta u - \beta \cdot \grad u}^2_{L^2(B_4)} \right) \\
		&+ C e^{2 \tau\Phi(r_3)} (\mu^2 + |\beta|^2) \norm{u}^2_{H^1(B_4)},
		\end{align*}
		for $\tau \ge \tau_0 + 2|\beta|^2/\mu^2$. We divide by $\mu^2$ and conclude by \cref{lem:log-convexity} with $p = 1 - \Phi(r_2) > 0$ and $q = \Phi(r_2) - \Phi(r_3) > 0$, followed by absorbing the $\tilde{Pe} = 1 + |\beta|/\mu$ factor into the exponential factor $e^{C\tilde{Pe}^2}$.
	\end{proof}
\end{corollary}

\def\h{\hbar}
\def\scl{\text{scl}}
We now shift down the Sobolev indices in \cref{cor:3_balls} by making a similar argument to that in Section 4 of \cite{DKSU} or Section 2.2 of \cite{BNO19}, based on semiclassical pseudodifferential calculus.
\begin{lemma}
	\label{lem:shifted3b}
	Let $x_0 \in \Omega$ and $0 < r_1 < r_2 < d(x_0,\partial \Omega)$.
	Define $B_j = B(x_0, r_j)$, $j=1,2$.
	Then there are $C > 0$ and $\kappa \in (0,1)$ such that for $\mu>0$, $\beta\in [L^\infty(\Omega)]^n$ and $u \in H^2(\Omega)$ it holds that
	\begin{align*}
	\norm{u}_{L^2(B_2)} \le C e^{C \tilde{Pe}^2} \left( \norm{u}_{L^2(B_1)} + \frac{1}{\mu} \norm{\mathcal{L} u}_{H^{-1}(\Omega)} \right)^\kappa
	\norm{u}_{L^2(\Omega)}^{1-\kappa},
	\end{align*}
	where $\tilde{Pe} = 1+|\beta|/\mu$ and $|\beta| = \norm{\beta}_{[L^\infty(\Omega)]^n}$.
	\begin{proof}
		Let $\h > 0$ be the semiclassical parameter
		that satisfies $\h = 1/\tau$, where $\tau$ is the parameter previously introduced in \cref{prop:carleman}. 
		We will make use of the theory of semiclassical pseudodifferential operators, which we briefly recall in \cref{appendixB} for the convenience of the reader. In particular we will use semiclassical Sobolev spaces with norms given by
		$$
		\norm{u}_{H_\scl^s(\R^n)} = \norm{J^s u}_{L^2(\R^n)},
		$$
		where the scale of the semiclassical Bessel potentials is defined by 
		$$
		J^s = (1-\h^2 \Delta)^{s/2}, \quad s \in \R.
		$$
		
		We will also use the following commutator and pseudolocal estimates, see \cref{appendixB}.
		Suppose that $\eta,\vartheta \in C_0^\infty(\R^n)$
		and that $\eta = 1$ near $\supp(\vartheta)$,
		and let $A_\psi, B_\psi$ be two semiclassical pseudodifferential 
		operators of orders $s, m$, respectively.
		Then for all $p,q,N \in \R$, there is $C>0$,
		\begin{align}
		\norm{[A_\psi,B_\psi] u}_{H_\scl^{p}(\R^n)} 
		&\le C \h \norm{u}_{H_\scl^{p+s+m-1}(\R^n)}, \label{commutator}
		\\
		\norm{(1-\eta) A_\psi \vartheta u}_{H_\scl^{p}(\R^n)}
		&\le C \h^N \norm{u}_{H_\scl^{q}(\R^n)}. \label{pseudolocality}
		\end{align}	
		
		Let $0<r_j < r_{j+1} < d(x_0,\partial \Omega),\, j=0,\ldots,4$ and $B_j=B(x_0,r_j)$, keeping $B_1, B_2$ unchanged. Let $\tilde{r}_j \in (r_{j-1}, r_j)$ and $\tilde{B}_j=B(x_0,\tilde{r}_j),\, j=0,\ldots,3$, where $r_{-1} = 0$. Choose $\rho \in C^\infty(\Omega)$ such that $\rho(x) = -d(x,x_0)$ outside $\tilde{B_0}$, and define $\phi = e^{\alpha \rho}$ for large enough $\alpha$. Consider $v \in C_0^\infty(B_5\setminus \tilde{B_0})$. As in \cref{appendix}, by taking $\ell = \phi / \h$ and $\sigma = \Delta \ell + 3 \alpha \lambda \phi / \h$, we obtain
		\begin{equation*}
		C \int_{\R^n} |e^{\phi / \h} \Delta (e^{-\phi / \h} v) |^2 \,\mathrm{d}x
		\ge 
		\int_{\R^n} (\h^{-1} |\nabla v|^2 
		+ \h^{-3} v^2
		- |\nabla v|^2
		- \h^{-2} v^2) \,\mathrm{d}x.
		\end{equation*}
		Scaling this with $\mu^2 \h^4$, we insert the convective term and obtain that
		\begin{align*}
		C \int_{\R^n} (\mu e^{\phi / \h} \h^2 \Delta (e^{-\phi / \h} v) 
		- e^{\phi / \h} \h^2 \beta \cdot \nabla(e^{-\phi / \h} v) )^2 \,\mathrm{d}x
		\end{align*}
		can be bounded from below by
		\begin{align*}
		\int_{\R^n} \h \mu^2 (\h^2 |\nabla v|^2 + v^2) \,\mathrm{d}x 
		- \int_{\R^n} \h^{2} \mu^2 (\h^2 |\nabla v|^2 + v^2) \,\mathrm{d}x
		- \int_{\R^n} (e^{\phi / \h} \h^2 \beta \cdot \nabla(e^{-\phi / \h} v))^2 \,\mathrm{d}x.
		\end{align*}
		Since
		\begin{equation*}
		e^{\phi / \h} \h^2 \beta \cdot \nabla(e^{-\phi / \h} v) = -\h (\beta \cdot \nabla \phi) v + \h^2 \beta \cdot \nabla v,
		\end{equation*} 
		introducing the conjugated operator $P v = -\h^2 e^{\phi / \h} \mathcal L (e^{-\phi / \h}v)$, the previous bound implies
		\begin{align*}
		C \norm{Pv}_{L^2(\R^n)}^2
		&\ge 
		\h \mu^2 \norm{v}_{H_\scl^1(\R^n)}^2 
		- \h^{2} \mu^2 \norm{v}_{H_\scl^1(\R^n)}^2
		- \h^2 |\beta|^2 \norm{v}_{H_\scl^1(\R^n)}^2.
		\end{align*}
		The last two terms in the right-hand side can be absorbed by the first one when
		\begin{equation}\label{h-condition}
		\h \le \frac12 \text{ and } \h \le \frac12 \frac{\mu^2}{|\beta|^2},
		\end{equation}
		thus obtaining
		\begin{equation}\label{semiclassical_H1}
		\sqrt \h \mu \norm{v}_{H_\scl^1(\R^n)} \le C \norm{P v}_{L^2(\R^n)}.
		\end{equation}
		
		Let now $\eta, \vartheta \in C_0^\infty(B_5\setminus \tilde{B_0})$ and suppose that $\vartheta=1$ near $B_4\setminus B_0$ and $\eta = 1$ near $\supp(\vartheta)$. Let also $\chi \in C_0^\infty(B_4 \setminus B_0)$ satisfy $\chi = 1$ in $B_3 \setminus \tilde{B}_1 $. Then there is $\h_0>0$ such that for $v=\chi w,\, w \in C^\infty(\Omega)$, and $\h<\h_0$,
		\begin{equation}\label{bessel}
		\norm{v}_{L^2(\R^n)}
		\le \norm{\eta J^{-1} v}_{H_\scl^{1}(\R^n)} 
		+ \norm{(1-\eta) J^{-1} \vartheta v}_{H_\scl^{1}(\R^n)}
		\le C \norm{\eta J^{-1} v}_{H_\scl^{1}(\R^n)}, 
		\end{equation}
		where we used \eqref{pseudolocality} to absorb one term by the left-hand side.  
		From \eqref{bessel} and \eqref{semiclassical_H1} we have 
		\begin{align}\label{shift_step1}
		\sqrt \h \mu \norm{v}_{L^2(\R^n)}\le C\sqrt \h \mu \norm{\eta J^{-1} v}_{H_\scl^1(\R^n)} 
		\le C \norm{P (\eta J^{-1} v)}_{L^2(\R^n)},
		\end{align}
		and the commutator estimate \eqref{commutator} gives
		\begin{align*}
		\norm{[P,\eta J^{-1}] v}_{L^2(\R^n)}
		\le C \h \mu \norm{v}_{L^2(\R^n)} + C \h^2 |\beta| \norm{v}_{H^{-1}_\scl(\R^n)}.
		\end{align*}
		Recalling the assumption \eqref{h-condition}, these terms can be absorbed by the left-hand side of \eqref{shift_step1}, obtaining
		\begin{align}\label{P-estimate}
		\sqrt \h \mu \norm{v}_{L^2(\R^n)}\le C 
		\norm{\eta J^{-1} (Pv)}_{L^2(\R^n)}
		\le C \norm{P v}_{H_\scl^{-1}(\R^n)}.
		\end{align}
			
		We now combine this estimate with the technique used to prove \cref{cor:3_balls}. Consider $u \in C^\infty(\R^n)$ and set $w = e^{\phi/\h} u$. Take $\psi \in C_0^\infty(\Omega)$ supported in $B_1 \cup (B_5 \setminus \tilde{B}_3)$ with $\psi = 1$ in $(\tilde{B}_1\setminus B_0) \cup (B_4\setminus B_3)$. Recall that $\chi \in C_0^\infty(B_4 \setminus B_0)$ satisfies $\chi = 1$ in $B_3 \setminus \tilde{B}_1 $. 
		Using \eqref{commutator} to bound the commutator
		$$
		\norm{[P,\chi] w}_{H_\scl^{-1}(\R^n)}
		\le
		\norm{[P,\chi] \psi w}_{H_\scl^{-1}(\R^n)}
		\le C \h (\mu+|\beta|) \norm{\psi w}_{L^2 (\R^n)},
		$$
		we obtain from \eqref{P-estimate} that
		\begin{align*}
		\sqrt \h \mu \norm{\chi w}_{L^2(\R^n)}
		&\le 
		C \norm{\chi P w}_{H_\scl^{-1}(\R^n)}
		+ C \h (\mu+|\beta|) \norm{\psi w}_{L^2 (\R^n)}.
		\end{align*}
		This leads to
		\begin{align*}
		\sqrt \h \mu \norm{\chi e^{\phi/\h} u}_{L^2(\R^n)}
		\le
		C \norm{\chi e^{\phi / \h} (\mu \Delta u - \beta \cdot \nabla u ) }_{H^{-1}(\R^n)}
		 + C \h (\mu+|\beta|) \norm{\psi e^{\phi/\h} u}_{L^2 (\R^n)},
		\end{align*}
		where we used the norm inequality $\norm{\cdot}_{H_\scl^{-1}(\R^n)} \le C \h^{-2} \norm{\cdot}_{H^{-1}(\R^n)}$. Letting $\Phi(r) = e^{-\alpha r}$ and using a similar argument as in the proof of \cref{cor:3_balls}, we find that
		\begin{align*}
		\mu e^{\Phi(r_2)/\h} \norm{u}_{L^2(B_2)}
		&\le
		C e^{1/\h} \left( (\mu+|\beta|) \norm{u}_{L^2(B_1)}
		+ \h^{-\frac32}\norm{ (\mu \Delta u - \beta \cdot \nabla u ) }_{H^{-1}(\Omega)} \right)
		\\& + C e^{\Phi(\tilde{r}_3)/\h} \h^{\frac12} (\mu+|\beta|) \norm{u}_{L^2 (\Omega)},
		\end{align*}		
		when $\h$ satisfies \eqref{h-condition} and is small enough.
		Absorbing the negative power of $\h$ in the exponential, we then use \cref{lem:log-convexity} and conclude by absorbing the $\tilde{Pe} = 1 + |\beta|/\mu$ factor into the exponential factor $e^{C\tilde{Pe}^2}$.
	\end{proof}
\end{lemma}

Making the additional coercivity assumption $\grad \cdot \beta \leq 0$, we can weaken the norms just in the right-hand side of \cref{cor:3_balls} by using the stability estimate for a well-posed convection-diffusion problem with homogeneous Dirichlet boundary conditions.

\begin{corollary}
	\label{cor:3_balls_impr}
	Let $x_0 \in \Omega$ and $0 < r_1 < r_2 < d(x_0,\partial \Omega)$.
	Define $B_j = B(x_0, r_j)$, $j=1,2$.
	Then there are $C > 0$ and $\kappa \in (0,1)$ such that for $\mu>0$, $\beta\in[W^{1,\infty}(\Omega)]^n$ having $\esssup_{\Omega} \grad \cdot \beta \leq 0$, and $u \in H^2(\Omega)$ it holds that
	\begin{align*}
	\norm{u}_{H^1(B_2)} &\le C e^{C \tilde{Pe}^2} \left( \norm{u}_{L^2(B_1)} + \frac{1}{\mu} \norm{\mathcal{L} u}_{H^{-1}(\Omega)}\right)^\kappa
	\left( \norm{u}_{L^2(\Omega)} + \frac{1}{\mu} \norm{\mathcal{L} u}_{H^{-1}(\Omega)} \right)^{1-\kappa},
	\end{align*}
	where $\tilde{Pe} = 1+|\beta|/\mu$ and $|\beta| = \norm{\beta}_{[L^\infty(\Omega)]^n}$.
	\begin{proof}
		Let the balls $B_0,B_3\subset \Omega$ such that $B_{j}\subset B_{j+1}$, for $j=0,2$. Consider the well-posed problem
		\begin{equation*}
		\mathcal{L} w = \mathcal{L} u \text{ in } B_3,\quad w=0 \text{ on } \partial B_3.
		\end{equation*}
		Since $\esssup_{\Omega} \grad \cdot \beta \leq 0$, as a consequence of the divergence theorem we have
		\begin{equation*}
		\|w\|_{H^1(B_3)} \le C \frac{1}{\mu} \|\mathcal{L} u\|_{H^{-1}(B_3)}.
		\end{equation*}
		
		Taking $v = u - w$, we have $\mathcal{L}v = 0$ in $B_3$. The stability estimate in \cref{cor:3_balls} used for $B_0,B_2,B_3$ reads as
		$$
		\norm{v}_{H^1(B_2)}
		\le
		C e^{C \tilde{Pe}^2} \norm{v}^\kappa_{H^1(B_0)} \norm{v}_{H^1(B_3)}^{1-\kappa},
		$$ 
		and the following estimates hold
		\begin{align*}
		\norm{u}_{H^1(B_2)}
		&\le
		\norm{v}_{H^1(B_2)} + \norm{w}_{H^1(B_2)}\\
		&\le
		C e^{C \tilde{Pe}^2} ( \norm{u}_{H^1(B_0)} + \frac{1}{\mu} \norm{\mathcal{L} u}_{H^{-1}(\Omega)} )^\kappa ( \norm{u}_{H^1(B_3)} + \frac{1}{\mu} \norm{\mathcal{L} u}_{H^{-1}(\Omega)} )^{1-\kappa}.
		\end{align*}
		Now we choose a cutoff function $\chi \in C^{\infty}_0(B_1)$ such that $\chi = 1$ in $B_0$. Then $\chi u$ satisfies
		$$
		\mathcal{L}(\chi u)=\chi \mathcal{L}u + [\mathcal{L}, \chi]u,\quad \chi u = 0 \text{ on } \p B_1,
		$$
		and we obtain
		\begin{align*}
		\norm{u}_{H^1(B_0)} \le \norm{\chi u}_{H^1(B_1)}
		&\le C \frac{1}{\mu} \left( \norm{ [\mathcal{L},\chi]u }_{H^{-1}(B_1)} + \norm{ \chi \mathcal{L}u }_{H^{-1}(B_1)} \right) \\
		&\le C \frac{1}{\mu} \left( (\mu + |\beta|)\norm{u}_{L^2(B_1)} + \norm{\mathcal{L}u}_{H^{-1}(\Omega)} \right)
		\end{align*}
		The same argument for $B_3 \subset \Omega$ gives
		$$
		\norm{u}_{H^1(B_3)} \le C \frac{1}{\mu} \left( (\mu + |\beta|)\norm{u}_{L^2(\Omega)} + \norm{\mathcal{L}u}_{H^{-1}(\Omega)} \right), 
		$$
		thus leading to the conclusion after absorbing the $\tilde{Pe} = 1 + |\beta|/\mu$ factor into the exponential factor $e^{C\tilde{Pe}^2}$.
	\end{proof}
\end{corollary}

\begin{remark}\label{remark-kappa}
	In the geometric setting of this section one can be more precise about the H\"older exponent $\kappa$ in the conditional stability estimates. For this we recall some known results for second-order elliptic equations: we refer to \cite[Theorem 2.1]{ARRV09} for the Laplace equation, and for the case including lower-order terms to \cite[Theorem 3]{MV12}. Let $u$ be a homogeneous solution of \eqref{eq:model_problem} with $f\equiv 0$. For a constant $C_{st}$ depending implicitly on the coefficients $\mu$ and $\beta$, the following three-ball inequality holds
	$$
	\|u\|_{L^2(B_2)} \le C_{st} \|u\|_{L^2(B_1)}^{\kappa} \|u\|_{L^2(B_3)}^{1-\kappa},
	$$
	where $B_j$ are concentric balls in $\Omega$ with increasing radii $r_j$. The constant $C_{st}$ does not depend on the radii $r_1,\,r_2$, but it does depend on $r_3$. The exponent $\kappa \in(0,1)$ is given by
	$$
	\kappa = \frac{\log \frac{r_3}{r_2}}{C_3\log \frac{r_2}{r_1} + \log \frac{r_3}{r_2}},
	$$
	where $C_3>0$ is a constant depending on $r_3$.
\end{remark}

\section{Finite element method}\label{sec:FEM}
Let $V_h$ denote the space of piecewise affine finite element functions
defined on a conforming computational mesh $\mathcal{T}_h = \{ K\}$. $\mathcal{T}_h$ consists
of shape regular triangular elements $K$ with diameter $h_K$ and is quasi-uniform. We define the
global mesh size by $h = \max_{K\in \mathcal{T}_h} h_K$. The
interior faces of the triangulation will be denoted by
$\mathcal{F}_i$, the jump of a quantity across a face $F$ by $\jump{\cdot}_F$, and the outward unit normal by $n$.

Let $\beta\in [W^{1,\infty}(\Omega)]^n$ and adopt the shorthand notation $|\beta| := \|\beta\|_{[L^\infty(\Omega)]^n}$. As already stated in \cref{sec:intro}, we consider the diffusion-dominated regime given by the low P\'eclet number
\begin{equation}\label{eq:Peclet_low}
Pe(h) := \frac{|\beta| h}{\mu} < 1.
\end{equation}
We will denote by $C$ a generic positive constant independent of the mesh size and the P\'eclet number.
Let $\pi_h:L^2(\Omega) \mapsto V_h$ denote the standard
$L^2$-projection on $V_h$, which for $k=1,2$ and $m=0,k-1$ satisfies
\begin{align*}
\norm{\pi_h u}_{H^m(\Omega)} &\le C \norm{u}_{H^m(\Omega)},\quad u\in H^m(\Omega), \\
\norm{u-\pi_h u}_{H^m(\Omega)} &\le C h^{k-m} \norm{u}_{H^k(\Omega)},\quad u\in H^k(\Omega) \label{interp}.
\end{align*}
We introduce the standard inner products with the induced norms  
$$
(v_h,w_h)_\Omega := \int_{\Omega} v_h w_h \,\mathrm{d}x,
$$
$$
\left<v_h,w_h\right>_{\partial \Omega} := \int_{\partial \Omega} v_h w_h \,\mathrm{d}s,
$$
and the following bilinear forms
$$
a_h(v_h,w_h) := (\beta \cdot \nabla  v_h, w_h )_\Omega+
( \mu \nabla v_h , \nabla w_h  )_{\Omega}-
\left< \mu \nabla v_h \cdot n , w_h \right>_{\partial \Omega},
$$
$$
s_\Omega(v_h,w_h) := \gamma \sum_{F \in \mathcal{F}_i} \int_{F}h (\mu +
|\beta|h) \jump{\nabla v_h \cdot n}_F \jump{\nabla w_h \cdot n}_F \,\mathrm{d}s,
$$
$$
s_{\omega}(v_h,w_h) := ((\mu + |\beta| h) v_h, w_h)_{ \omega},
$$
$$
s(v_h,w_h) := s_\Omega(v_h,w_h) + s_{\omega}(v_h,w_h),
$$
and
$$
s_*(v_h,w_h) :=\gamma_*( \left<(\mu h^{-1} + |\beta|)
v_h,w_h\right>_{\partial \Omega}+ (\mu \nabla v_h, \nabla w_h )_\Omega+ s_\Omega(v_h,w_h)).
$$
The terms $s_\Omega$ and $s_*$ are stabilizing terms, while the term $s_\omega$ is aimed for data assimilation. After scaling with the coefficients in the above forms, Lemma 2 in \cite{BHL18} writes as
\begin{equation}\label{eq:Poincare}
\|(\mu^{\frac12} h + |\beta|^{\frac12} h^{\frac32})
v_h\|_{H^1(\Omega)} \leq C \gamma^{-\frac12} s(v_h,v_h)^{\frac12}, \quad \forall v_h \in V_h,
\end{equation}
and Lemma 2 in \cite{BO18} gives the jump inequality
\begin{equation}\label{jumpineq}
s_\Omega(\pi_h u,\pi_h u) \le C \gamma (\mu + |\beta|h) h^2 |u|^2_{H^2(\Omega)},\quad \forall u\in H^2(\Omega).
\end{equation}

The parameters $\gamma$ and $\gamma_*$ in $s_\Omega$ and $s_*$, respectively, are fixed at the implementation level and, to alleviate notation, our analysis covers the choice $\gamma=1=\gamma_*$.

We can then use the general framework in \cite{Bur13} to write the finite element method for unique continuation subject to
\eqref{eq:model_problem} as follows. Consider a discrete Lagrange multiplier $z_h\in V_h$ and aim to find the saddle points of the functional
\begin{align*}
L_h(u_h,z_h) :&= \frac12 s_\omega (u_h - \tilde U_\omega, u_h - \tilde U_\omega) + a_h(u_h,z_h) - (f,z_h)_\Omega \\
&+ \frac12 s_\Omega(u_h,u_h) - \frac12 s_*(z_h,z_h),
\end{align*}
where we recall that $\tilde U_\omega = u\vert_{\omega} + \delta$ and $u\in H^2(\Omega)$ is a solution to \eqref{eq:model_problem}. The Euler-Lagrange equations for $L_h$ lead to the following discrete problem: find $(u_h,z_h)  \in [V_h]^2$ such that
\begin{equation}\label{eq:model_FEM}
\left\{\begin{array}{rcl}
a_h(u_h,w_h) - s_*(z_h,w_h) &=& (f,w_h)_\Omega\\
a_h(v_h,z_h) + s(u_h,v_h) &=&s_{\omega}(\tilde U_\omega,v_h)
\end{array} \right.
,\quad \forall (v_h,w_h)  \in [V_h]^2,
\end{equation}

We observe that by the ill-posed character of the problem, only the stabilization operators $s_\Omega$ and $s_*$ provide some stability to the discrete system, and the corresponding system matrix is expected to be ill-conditioned. To quantify this effect we first prove an upper bound on the condition number.

\begin{proposition}\label{prop:unique}
	The finite element formulation \eqref{eq:model_FEM} has a unique solution  $(u_h,z_h) \in [V_h]^2$ and the Euclidean condition number $\mathcal{K}_2$ of the system matrix  satisfies $$\mathcal{K}_2 \le C h^{-4}.$$ 
	\begin{proof}
		We write \eqref{eq:model_FEM} as the linear system $A[(u_h,z_h),(v_h,w_h)] = (f,w_h)_\Omega + s_{\omega}(\tilde U_\omega,v_h)$, for all $(v_h,w_h)  \in [V_h]^2$, where $$A[(u_h,z_h),(v_h,w_h)] := a_h(u_h,w_h) - s_*(z_h,w_h) + a_h(v_h,z_h) + s(u_h,v_h).$$
		Since $A[(u_h,z_h),(u_h,-z_h)] = s(u_h,u_h) + s_*(z_h,z_h)$, using \eqref{eq:Poincare} the following inf-sup condition holds
		$$
		 \Psi_h:= \adjustlimits \inf_{(u_h,z_h)\in [V_h]^2} \sup_{(v_h,w_h)\in [V_h]^2} \frac{A[(u_h,z_h),(v_h,w_h)]}{\|(u_h,z_h)\|_{L^2(\Omega)}\|(v_h,w_h)\|_{L^2(\Omega)}} \ge C \mu (1+Pe(h)) h^2.
		$$
		This provides the existence of a unique solution for the linear system. We use \cite[Theorem 3.1]{EG06} to estimate the condition number by
		\begin{equation}\label{eq:condition-ineq}
		\mathcal{K}_2 \le C \frac{\Upsilon_h}{\Psi_h},
		\end{equation}
		where 
		$$
		\Upsilon_h:= \adjustlimits \sup_{(u_h,z_h)\in [V_h]^2} \sup_{(v_h,w_h)\in [V_h]^2} \frac{A[(u_h,z_h),(v_h,w_h)]}{\|(u_h,z_h)\|_{L^2(\Omega)}\|(v_h,w_h)\|_{L^2(\Omega)}}.
		$$
		We recall the following discrete inverse inequality, see for instance \cite[Lemma 1.138]{EG04},
		\begin{equation}\label{eq:inverse-ineq}
		\| \nabla v_h\|_{L^2(K)} \le C h^{-1} \|v_h\|_{L^2(K)},\quad \forall v_h \in \mathbb{P}_1(K).
		\end{equation}
		We also recall the following continuous trace inequality, see for instance \cite{MS99},
		\begin{equation}\label{eq:trace_cont}
		\|v\|_{L^2(\partial K)} \le C( h^{-\frac12} \|v\|_{L^2(K)} + h^{\frac12} \|\nabla v\|_{L^2(K)}),\quad \forall v \in H^1(K),
		\end{equation}
		and the discrete one
		\begin{equation}\label{eq:trace_grad}
		\|\nabla v_h \cdot n\|_{L^2(\partial K)} \le C h^{-\frac12} \|\nabla v_h\|_{L^2(K)},\quad \forall v_h \in \mathbb{P}_1(K).
		\end{equation}
		Using the Cauchy-Schwarz inequality together with \eqref{eq:trace_grad} and \eqref{eq:inverse-ineq} we get
		\begin{align*}
		s_\Omega(u_h,v_h) &= \gamma \mu (1+Pe(h)) \sum_{F \in \mathcal{F}_i} \int_{F}h \jump{\nabla u_h \cdot n}_F \jump{\nabla v_h \cdot n}_F ~\mbox{d}s
		\\&\le C \mu (1+Pe(h)) h^{-2} \| u_h\|_{L^2(\Omega)} \| v_h\|_{L^2(\Omega)},
		\end{align*}
		hence
		\begin{equation*}
		s(u_h,v_h) \le C \mu (1+Pe(h)) h^{-2} \| u_h\|_{L^2(\Omega)} \| v_h\|_{L^2(\Omega)}.
		\end{equation*}
		Combining this with the Cauchy-Schwarz inequality and the inequalities \eqref{eq:inverse-ineq} and \eqref{eq:trace_cont}, we obtain
		\begin{equation*}
		-s_*(z_h,w_h) \le C \mu (1+Pe(h)) h^{-2} \| z_h\|_{L^2(\Omega)} \| w_h\|_{L^2(\Omega)}.
		\end{equation*}
		Again due to the Cauchy-Schwarz inequality, and trace and inverse inequalities,  we have
		\begin{align*}
		a_h(u_h,w_h) &= (\beta \cdot \nabla  u_h, w_h )_\Omega + \mu \sum_{F \in \mathcal{F}_i} \int_{F}h \jump{\nabla u_h \cdot n}_F w_h \,\mathrm{d}s
		\\&\le C \mu (1+Pe(h)) h^{-2} \| u_h\|_{L^2(\Omega)} \| w_h\|_{L^2(\Omega)},
		\end{align*}
		Collecting the above estimates we have
		$
		\Upsilon_h \le C \mu (1+Pe(h)) h^{-2},
		$
		and we conclude by \eqref{eq:condition-ineq}.
	\end{proof}
\end{proposition}

\subsection{Error estimates for the weakly consistent regularization}
The error analysis proceeds in two main steps:
\begin{enumerate}[(i)]
	\item First we prove that the stabilizing terms and the data fitting term must vanish at an optimal rate for smooth solutions, with constant independent of the physical stability (\cref{prop:conv_stab}).
	\item Then we show that the residual of the PDE is bounded by the stabilizing terms and the data fitting term. Using this result together with the first step and the continuous stability estimates in \cref{sec:stability_estimates}, we prove $L^2$- and $H^1$-convergence results (Theorems \ref{L^2-error} and \ref{H^1-error}).
\end{enumerate}
To quantify stabilization and data fitting for $(v_h,w_h) \in [V_h]^2$ we introduce the norm 
$$
\|(v_h,w_h)\|^2_s:= s(v_h,v_h) + s_*(w_h,w_h).
$$
We also define the ``continuity norm'' on $H^{\frac32+\epsilon}(\Omega)$, for any $\epsilon>0$,
$$
\| v \|_\sharp:= \| |\beta|^\frac12 h^{-\frac12} v \|_{\Omega} + \|\mu^{\frac12} \nabla
v\|_\Omega + \| \mu^{\frac12} h^{\frac12} \nabla v \cdot n\|_{\partial \Omega}.
$$
Using standard approximation properties and the trace inequality \eqref{eq:trace_cont}, we have 
\begin{align*}
\| u - \pi_h u \|_\sharp \le C ( \mu^{\frac12} h + |\beta|^\frac12 h^{\frac32} ) |u|_{H^2(\Omega)}.
\end{align*}
Using \eqref{jumpineq} and interpolation
\begin{align*}
\|(u - \pi_h u,0)\|^2_s &= s(u -\pi_h u,u -\pi_h u) = s_\Omega(\pi_h u,\pi_h u) + s_\omega(u -\pi_h u,u -\pi_h u)
\\
&\le C (\mu h^2 + |\beta| h^3) |u|^2_{H^2(\Omega)},
\end{align*}
where we used that $s_\Omega(u,v_h) = 0$, since $u \in H^2(\Omega)$. Hence it follows that for $u \in H^2(\Omega)$
\begin{equation}\label{eq:approx}
\|(u - \pi_h u,0)\|_s + \| u - \pi_h u \|_\sharp \le C (\mu^{\frac12} h +
|\beta|^{\frac12} h^{\frac32}) |u|_{H^2(\Omega)}.
\end{equation}
Observe that, when $Pe(h)<1$, the first term dominates and the estimate is $O(h)$, whereas when $Pe(h)>1$ the bound is $O(h^{\frac32})$. We note in passing that the same estimates hold for the nodal interpolant.

\begin{lemma}[Consistency]\label{lem:consist}
	Let $u \in H^2(\Omega)$ be a solution to \eqref{eq:model_problem} and $(u_h,z_h) \in [V_h]^2$ the solution to \eqref{eq:model_FEM}, then
	$$
	a_h(\pi_h u - u_h,w_h) + s_*(z_h,w_h) = a_h(\pi_h u - u,w_h),
	$$
	and
	$$
	-a_h(v_h,z_h) + s(\pi_h u - u_h,v_h) = s_\Omega(\pi_h u - u,v_h) +
	s_\omega(\pi_h u - \tilde U_\omega,v_h),
	$$
	for all $(v_h,w_h)\in [V_h]^2$.
	\begin{proof}
		The first claim follows from the definition of $a_h$, since
		\begin{equation*}
		a_h(u_h,w_h) - s_*(z_h,w_h) = (f,w_h)_\Omega = (\beta \cdot \nabla u -
		\mu \Delta u, w_h)_\Omega = a_h(u,w_h),
		\end{equation*}
		where in the last equality we integrated by parts.
		The second claim follows similarly from
		$$
		a_h(v_h,z_h) + s(u_h,v_h) = 
		s_\omega(\tilde U_\omega,v_h), 
		$$
		leading to
		\begin{align*}
		-a_h(v_h,z_h) + s(\pi_h u - u_h,v_h) &= s(\pi_h u,v_h) - s_\omega(\tilde
		U_\omega,v_h)  \\
		&=  s_\Omega(\pi_h u - u,v_h) + s_\omega(\pi_h u - \tilde U_\omega,v_h).
		\end{align*}
	\end{proof}
\end{lemma}
\begin{lemma}[Continuity]\label{lem:cont}
	Assume the low P\'eclet regime \eqref{eq:Peclet_low} and that $|\beta|_{1,\infty} \le C |\beta|$. Let $v \in H^2(\Omega)$ and $w_h \in V_h$, then
	$$
	a_h(v,w_h) \leq C \| v \|_\sharp \|(0,w_h)\|_s.
	$$
	\begin{proof}
		Writing out the terms of $a_h$ and integrating by parts in the
		advective term leads to
		$$
		a_h(v,w_h) = - (v,\beta \cdot \nabla w_h)_\Omega - (v \nabla \cdot \beta, w_h)_\Omega
		+ \left<v \beta \cdot n,w_h \right>_{\partial \Omega} + ( \mu \nabla v , \nabla w_h  )_{\Omega}-
		\left< \mu \nabla v \cdot n , w_h \right>_{\partial \Omega}.
		$$
		Using the Cauchy-Schwarz inequality and the trace inequality \eqref{eq:trace_cont} for $v$, we see that
		$$
		\left<v \beta \cdot n,w_h \right>_{\partial \Omega}+ ( \mu \nabla v, \nabla w_h  )_{\Omega}-
		\left< \mu \nabla v \cdot n, w_h \right>_{\partial \Omega} \leq C \| v \|_\sharp \|(0,w_h)\|_s.
		$$
		By the Cauchy-Schwarz inequality and a discrete Poincar\'e inequality for $w_h$, see e.g. \cite{Bre03}, we bound
		$$
		-(v \nabla \cdot \beta, w_h)_\Omega \le C |\beta|_{1,\infty} \|v\|_\Omega \|w_h\|_\Omega \le C \frac{|\beta|_{1,\infty}}{|\beta|} Pe(h)^\frac12 \|v\|_\sharp  \|(0,w_h)\|_s.
		$$ 
		Under the assumption $|\beta|_{1,\infty} \le C |\beta|$, we get
		$$
		-(v \nabla \cdot \beta, w_h)_\Omega \le C Pe(h)^\frac12 \|v\|_\sharp \|(0,w_h)\|_s.
		$$
		We bound the remaining term by
		\begin{align*}
		-(v,\beta \cdot \nabla w_h)_\Omega \le |\beta|^\frac12 h^{\frac12} \|v\|_\sharp \|\nabla w_h\|_\Omega \le C Pe(h)^\frac12 \|v\|_\sharp \|(0,w_h)\|_s.
		\end{align*}
		Finally, exploiting the low P\'eclet regime $Pe(h)<1$, we obtain the conclusion.
	\end{proof}
\end{lemma}

\begin{proposition}[Convergence of regularization]\label{prop:conv_stab}
	Assume the low P\'eclet regime \eqref{eq:Peclet_low} and that $|\beta|_{1,\infty} \le C |\beta|$. Let $u \in H^2(\Omega)$ be a solution to \eqref{eq:model_problem}
	and $(u_h,z_h) \in [V_h]^2$ the solution to \eqref{eq:model_FEM}, then
	$$
	\|(\pi_h u - u_h,z_h)\|_s \le C (\mu^{\frac12} h + |\beta|^{\frac12}
	h^{\frac32}) (|u|_{H^2(\Omega)} + h^{-1} \|\delta\|_\omega).
	$$
	\begin{proof}
		Denoting $e_h = \pi_h u - u_h$, it holds by definition that
		$$
		\|(e_h,z_h)\|_s^2 = a_h(e_h,z_h)  + s_*(z_h,z_h) - a_h(e_h,z_h) + s(e_h,e_h).
		$$
		Using both claims in \cref{lem:consist} we may write
		$$
		\|(e_h,z_h)\|_s^2 = a_h(\pi_h u - u,z_h) +s_\Omega(\pi_h u -
		u,e_h)+s_\omega(\pi_h u - \tilde U_\omega,e_h).
		$$
		\cref{lem:cont} gives the bound
		$$
		a_h(\pi_h u - u,z_h) \leq C \|\pi_h u - u\|_\sharp \|(0,z_h)\|_s.
		$$
		The other terms are simply bounded using the Cauchy-Schwarz inequality as follows
		\begin{equation*}
		s_\Omega(\pi_h u -
		u,e_h)+s_\omega(\pi_h u - \tilde U_\omega,e_h)
		\leq (\|(\pi_h u -
		u,0)\|_s + (\mu^{\frac12} + |\beta|^{\frac12} h^{\frac12})\|\delta\|_\omega) \|(e_h,0)\|_s.
		\end{equation*}
		Collecting the above bounds we have
		\begin{equation*}
		\|(e_h,z_h)\|_s^2
		\leq C(\|\pi_h u - u\|_\sharp+\|(\pi_h u -
		u,0)\|_s + (\mu^{\frac12} + |\beta|^{\frac12} h^{\frac12}) \|\delta\|_\omega) \|(e_h,z_h)\|_s,
		\end{equation*}
		and the claim follows by applying the approximation \eqref{eq:approx}.
	\end{proof}
\end{proposition}

\begin{lemma}[Covergence of the convective term]\label{lem:conv_convection}
Assume the low P\'eclet regime \eqref{eq:Peclet_low} and that $|\beta|_{1,\infty} \le C |\beta|$. Let $u \in H^2(\Omega)$ be a solution to \eqref{eq:model_problem}, $(u_h,z_h) \in [V_h]^2$ the solution to \eqref{eq:model_FEM} and $w\in H^1_0(\Omega)$, then
$$
(\beta \cdot \nabla u_h, w-\pi_h w)_\Omega \le C (\mu + |\beta|) ( h \|u\|_{H^2(\Omega)} + \|\delta\|_\omega)\|w\|_{H^1(\Omega)},
$$
\begin{proof}
	Denote by $\beta_h\in [V_h]^n$ a piecewise linear approximation of $\beta$ that is $L^\infty$-stable and for which
	\begin{equation*}
	\|\beta - \beta_h\|_{0,\infty} \leq C h |\beta|_{1,\infty}, 
	\end{equation*}
	and recall the approximation estimate in \cite[Theorem 2.2]{Bur05}
	\begin{equation}\label{eq:convective_approximation}
	\inf_{x_h \in V_h} \|h^{\frac12} (\beta_h \cdot \nabla u_h - x_h)\|_\Omega
	\leq C \left(\sum_{F \in \mathcal{F}_i} \|h \jump{\beta_h \cdot \nabla u_h}\|_{F}^2\right)^{\frac12}
	\leq C |\beta|^\frac12 s_\Omega(u_h,u_h)^{\frac12}.
	\end{equation}
	We also use \cref{prop:conv_stab} and the jump inequality \eqref{jumpineq} to estimate
	\begin{align*}
	s_{\Omega}(u_h,u_h)^{\frac12} &\le s_{\Omega}(u_h-\pi_h u,u_h-\pi_h u)^{\frac12} + s_{\Omega}(\pi_h u,\pi_h u)^{\frac12}
	\\&
	\le C (\mu^{\frac12} h + |\beta|^{\frac12}
	h^{\frac32}) (|u|_{H^2(\Omega)} + h^{-1} \|\delta\|_\omega)
	+ C (\mu^\frac12 + |\beta|^\frac12 h^\frac12) h|u|_{H^2(\Omega)},
	\end{align*}
	obtaining
	\begin{align}\label{eq:conv_jump}
	s_{\Omega}(u_h,u_h)^{\frac12}
	\le C (\mu^{\frac12} h + |\beta|^{\frac12}
	h^{\frac32}) (|u|_{H^2(\Omega)} + h^{-1} \|\delta\|_\omega).
	\end{align}
	We now write
	$$
	(\beta \cdot \nabla u_h, w-\pi_h w)_\Omega  = (\beta_h \cdot \nabla u_h, w-\pi_h w)_\Omega  + ((\beta-\beta_h) \cdot \nabla u_h, w-\pi_h w)_\Omega,
	$$
	and using orthogonality, \eqref{eq:convective_approximation}, \eqref{eq:conv_jump}, interpolation and \eqref{eq:Peclet_low}, we bound the first term by
	\begin{align*}
	(\beta_h \cdot \nabla u_h, w-\pi_h w)_\Omega &\le C |\beta|^\frac12 h^{-\frac12} s_\Omega(u_h,u_h)^{\frac12} h \|w\|_{H^1(\Omega)}
	\\&\le C |\beta|^\frac12 h^\frac12 (\mu^{\frac12} + |\beta|^{\frac12}
	h^{\frac12}) (h|u|_{H^2(\Omega)} + \|\delta\|_\omega) \|w\|_{H^1(\Omega)}
	\\&\le C (\mu + |\beta| h) (h|u|_{H^2(\Omega)} + \|\delta\|_\omega) \|w\|_{H^1(\Omega)}.
	\end{align*}
	We now use the Poincar\'e-type inequality \eqref{eq:Poincare} and interpolation to bound the second term
	\begin{align*}
	((\beta-\beta_h) \cdot \nabla u_h, w-\pi_h w)_\Omega &\le C h^2 |\beta|_{1,\infty} \|\nabla u_h\|_\Omega \|w\|_{H^1(\Omega)}
	\\&\le  C h |\beta|_{1,\infty} (\mu^\frac12 + |\beta|^\frac12 h^\frac12)^{-1} s(u_h,u_h)^{\frac12} \|w\|_{H^1(\Omega)}
	\\&\le C h |\beta|_{1,\infty} ( h |u|_{H^2(\Omega)} + \|u\|_\Omega + \|\delta\|_\omega ) \|w\|_{H^1(\Omega)}
	\\&\le C h |\beta|_{1,\infty} ( \|u\|_{H^2(\Omega)} + \|\delta\|_\omega ) \|w\|_{H^1(\Omega)}
	\end{align*}
	since due to \cref{prop:conv_stab} and inequality \eqref{jumpineq}
	\begin{align*}
	s(u_h,u_h)^\frac12 &\le s(u_h-\pi_h u,u_h-\pi_h u)^\frac12 + s_\Omega(\pi_h u,\pi_h u)^{\frac12} + s_\omega(\pi_h u,\pi_h u)^{\frac12}
	\\&\le C (\mu^{\frac12} + |\beta|^{\frac12}
	h^{\frac12}) ( h |u|_{H^2(\Omega)} + \|\delta\|_\omega + \|u\|_\Omega ).
	\end{align*}
	Under the assumption $|\beta|_{1,\infty} \le C |\beta|$, we collect the above bounds to get
	$$
	(\beta \cdot \nabla u_h, w-\pi_h w)_\Omega \le C (\mu + |\beta|) ( h \|u\|_{H^2(\Omega)} + \|\delta\|_\omega)\|w\|_{H^1(\Omega)}.
	$$
\end{proof}
\end{lemma}

We now combine these results with the conditional stability estimates from \cref{sec:stability_estimates} to obtain error bounds for the discrete solution. For this purpose, we consider an open bounded set $B\subset \Omega$ that contains the data region $\omega$ such that $B\setminus \omega$ does not touch the boundary of $\Omega$. Then the estimates in \cref{lem:shifted3b} and \cref{cor:3_balls_impr} hold true by a covering argument, see e.g. \cite{MV12}, and we obtain local error estimates in $B$. For global unique continuation from $\omega$ to the entire $\Omega$, however, the stability deteriorates and it is of a different nature: the modulus of continuity for the given data is not of H\"older type $|\cdot|^\kappa$ any more, but of a logarithmic kind $|\log(\cdot)|^{-\kappa}$.

\begin{theorem}[$L^2$-error estimate]\label{L^2-error}
	Assume the low P\'eclet regime \eqref{eq:Peclet_low} and that $|\beta|_{1,\infty} \le C |\beta|$. Consider $\omega \subset B \subset \Omega$ such that $\overline{B\setminus \omega} \subset \Omega$. Let $u \in H^2(\Omega)$ be a solution to \eqref{eq:model_problem} and $(u_h,z_h) \in [V_h]^2$ the solution to \eqref{eq:model_FEM}, then there is $\kappa\in(0,1)$ such that
	\begin{equation*}
	\|u-u_h\|_{L^2(B)} \le C h^{\kappa} e^{C \tilde{Pe}^2} (\|u\|_{H^2(\Omega)} + h^{-1} \|\delta\|_\omega),
	\end{equation*}
	where $\tilde{Pe} = 1+|\beta|/\mu$.
	\begin{proof}
		Let us consider the residual defined by $\left< r,w \right> = a_h(u_h,w)-\left<f,w \right>$, for $w\in H^1_0(\Omega)$. Using \eqref{eq:model_FEM} we obtain
		\begin{align*}
		\left< r,w \right> &= a_h(u_h,w-\pi_h w) - \left< f,w-\pi_h w \right> + a_h(u_h,\pi_h w) - \left< f,\pi_h w \right>
		\\&= a_h(u_h,w-\pi_h w) - \left< f,w-\pi_h w \right> + s_*(z_h,\pi_h w).
		\end{align*}
		We split the first term in the right-hand side into convective and non-convective terms, and for the latter we integrate by parts on each element $K$ and use Cauchy-Schwarz followed by the trace inequality \eqref{eq:trace_cont} to get
		\begin{align*}
		&( \mu \nabla u_h , \nabla (w-\pi_h w)  )_{\Omega}-
		\left< \mu \nabla u_h \cdot n , w-\pi_h w \right>_{\partial \Omega} = \sum_{F \in \mathcal{F}_i} \int_{F}\mu\jump{\nabla u_h \cdot n}_F (w-\pi_h w) \,\mathrm{d}s
		\\ &\le C \mu (\mu + |\beta|h)^{-\frac12} s_\Omega(u_h,u_h)^{\frac12} (h^{-1} \|w-\pi_h w\|_{L^2(\Omega)} + \|w-\pi_h w\|_{H^1(\Omega)}).
		\end{align*}
		Using \eqref{eq:conv_jump} and interpolation we obtain
		\begin{align*}
		( \mu \nabla u_h , \nabla (w-\pi_h w)  )_{\Omega}-\left< \mu \nabla u_h \cdot n , w-\pi_h w \right>_{\partial \Omega}
		\le C \mu ( h |u|_{H^2(\Omega)} + \|\delta\|_\omega ) \|w\|_{H^1(\Omega)}.
		\end{align*}
		We bound the convective term in $a_h(u_h,w-\pi_h w)$ by \cref{lem:conv_convection}, hence obtaining
		\begin{equation*}
		a_h(u_h,w-\pi_h w) \le C (\mu + |\beta|) ( h \|u\|_{H^2(\Omega)} + \|\delta\|_\omega) \|w\|_{H^1(\Omega)}.
		\end{equation*}
		The next term in the residual is bounded by
		$$
		\left< f,w-\pi_h w \right> \le \| f \|_{L^2(\Omega)} \| w-\pi_h w \|_{L^2(\Omega)} \le C h \| f \|_{L^2(\Omega)} \| w \|_{H^1(\Omega)}.
		$$
		The last term left to bound from the residual is
		\begin{align*}
		s_*(z_h,\pi_h w) \le \| (0,z_h) \|_s \| (0, \pi_h w)\|_s,
		\end{align*}
		and using \eqref{eq:trace_grad} for the jump term, together with the $H^1$-stability of $\pi_h$, we see that
		\begin{align*}
		\| (0, \pi_h w)\|_s &\le C( \mu^\frac12 \|\nabla(\pi_h w)\|_\Omega + (\mu^{\frac12} + |\beta|^{\frac12} h^{\frac12}) \|\nabla(\pi_h w)\|_\Omega + (\mu h^{-1} + |\beta|)^\frac12 \|\pi_h w\|_{\partial \Omega} )
		\\
		&\le C (\mu^{\frac12} + |\beta|^{\frac12} h^{\frac12} ) \|w\|_{H^1(\Omega)},
		\end{align*}
		where for the boundary term we used that $w|_{\partial\Omega} = 0$ together with interpolation and \eqref{eq:trace_cont}. Bounding $\| (0,z_h) \|_s$ by \cref{prop:conv_stab}, we get
		\begin{align*}
		s_*(z_h,\pi_h w) \le C (\mu + |\beta|h) ( h |u|_{H^2(\Omega)} + \|\delta\|_\omega ) \|w\|_{H^1(\Omega)}.
		\end{align*}
		Collecting the above estimates we bound the residual norm by
		\begin{align*}
		\|r\|_{H^{-1}(\Omega)} &\le C (\mu + |\beta|)(h \|u\|_{H^2(\Omega)} + \|\delta\|_\omega) + C h \| f \|_{L^2(\Omega)}
		\\&\le C (\mu + |\beta|) (h \|u\|_{H^2(\Omega)} + \|\delta\|_\omega).
		\end{align*}
		We now use the stability estimate in \cref{lem:shifted3b} to write
		\begin{equation*}
		\|u-u_h\|_{L^2(B)} \le C e^{C \tilde{Pe}^2} \left( \|u-u_h\|_{L^2(\omega)} + \frac{1}{\mu} \|r\|_{H^{-1}(\Omega)} \right)^{\kappa} \|u-u_h\|_{L^2(\Omega)}^{1-\kappa}.
		\end{equation*}
		By \cref{prop:conv_stab} we have
		\begin{align*}
		\|u-u_h\|_{L^2(\omega)} &\le \|u-\pi_h u\|_{L^2(\omega)} + \|u_h - \pi_h u\|_{L^2(\omega)}
		\\&\le C h^2 |u|_{H^2(\Omega)} + C h |u|_{H^2(\Omega)} + C \|\delta\|_\omega.
		\\&\le C (h |u|_{H^2(\Omega)} + \|\delta\|_\omega).
		\end{align*}
		Using \eqref{eq:Poincare} and \cref{prop:conv_stab} again, we bound
		\begin{align*}
		\|u-u_h\|_{L^2(\Omega)} &\le \|u-\pi_h u\|_{L^2(\Omega)} + \|u_h-\pi_h u\|_{L^2(\Omega)}
		\\ &\le  C h^2 |u|_{H^2(\Omega)} + C (\mu^{\frac12} h + |\beta|^{\frac12} h^{\frac32})^{-1} s(u_h-\pi_h u,u_h-\pi_h u)^{\frac12}
		\\ &\le C (|u|_{H^2(\Omega)} + h^{-1}\|\delta\|_\omega).
		\end{align*}
		Hence we conclude by
		\begin{align*}
		\|u-u_h\|_{L^2(B)} &\le C e^{C \tilde{Pe}^2} \left( h \|u\|_{H^2(\Omega)} + \|\delta\|_\omega \right)^{\kappa} \left( |u|_{H^2(\Omega)} + h^{-1} \|\delta\|_\omega \right)^{1-\kappa}
		\\ &\le C e^{C \tilde{Pe}^2} h^{\kappa} (\|u\|_{H^2(\Omega)} + h^{-1} \|\delta\|_\omega),
		\end{align*}
		where we have absorbed the $\tilde{Pe} = 1 + |\beta|/\mu$ factor into the exponential factor $e^{C\tilde{Pe}^2}$.
	\end{proof}
\end{theorem}

\begin{remark}
	Let us briefly discuss the effect of decreasing the size of the data region $\omega$ by considering the case of balls, that is $\omega = B(x_0, r_1)$ and $B = B(x_0, r_2)$, with $x_0\in \Omega$ and $r_1<r_2$. Notice from \cref{remark-kappa} that the exponent $\kappa$ is an increasing function of the radius $r_1$ and that decreasing the size of the data region $\omega$ implies that the convergence rate $h^\kappa$ decreases as well. Bounding the radius $r_2$ away from zero and letting $r_1\to 0$ implies that the exponent $\kappa\to 0$.
	The continuum three-ball inequality then becomes the trivial inequality $\|u\|_{L^2(B)} \le \|u\|_{L^2(\Omega)}$ and the method does not converge any more.
\end{remark}

\begin{theorem}[$H^1$-error estimate]\label{H^1-error}
	Assume the low P\'eclet regime \eqref{eq:Peclet_low} and that $|\beta|_{1,\infty} \le C |\beta|$ and $\esssup_{\Omega} \grad \cdot \beta \leq 0$. Consider $\omega \subset B \subset \Omega$ such that $\overline{B\setminus \omega} \subset \Omega$. Let $u \in H^2(\Omega)$ be a solution to \eqref{eq:model_problem}, and $(u_h,z_h) \in [V_h]^2$ the solution to \eqref{eq:model_FEM}, then there is $\kappa\in(0,1)$ such that
	\begin{equation*}
	\|u-u_h\|_{H^1(B)} \le C h^{\kappa} e^{C \tilde{Pe}^2} (\|u\|_{H^2(\Omega)} + h^{-1} \|\delta\|_\omega),
	\end{equation*}
	where $\tilde{Pe} = 1+|\beta|/\mu$.
	\begin{proof}
		Letting $e_h = u-u_h$, we combine the proof of \cref{L^2-error} with the stability estimate in \cref{cor:3_balls_impr} to obtain
		\begin{align*}
		\norm{e_h}_{H^1(B)} &\le C e^{C \tilde{Pe}^2} \left( \norm{e_h}_{L^2(\omega)} + \frac{1}{\mu} \norm{r}_{H^{-1}(\Omega)}\right)^\kappa
		\left( \norm{e_h}_{L^2(\Omega)} + \frac{1}{\mu} \norm{r}_{H^{-1}(\Omega)} \right)^{1-\kappa}
		\\&\le C e^{C \tilde{Pe}^2} h^{\kappa} (\|u\|_{H^2(\Omega)} + h^{-1} \|\delta\|_\omega).
		\end{align*}
	\end{proof}
\end{theorem}

\section{Numerical experiments}
We illustrate the theoretical results with some numerical examples. The implementation of the stabilized FEM \eqref{eq:model_FEM} has been carried out in FreeFem++ \cite{Hec12} on uniform triangulations with alternating left and right diagonals. The mesh size is taken as the inverse square root of the number of nodes. The parameters in $s_\Omega$ and $s_*$ are set to $\gamma = 10^{-5}$ and $\gamma_* = 1$. We also rescale the boundary term in $s_*$ by the factor 50, drawing on results from different numerical experiments. In this section we denote $e_h = \pi_h u - u_h$.

We consider $\Omega$ to be the unit square and the exact solution with global unit $L^2$-norm
\begin{equation*}\label{eq:exact_sol}
u(x,y) = 30x(1-x)y(1-y).
\end{equation*}
We take the diffusion coefficient $\mu = 1$ and investigate two cases for the convection field: the coercive case of the constant field
$$\beta_c=(1,0),$$
and the case
$$\beta_{nc} = 100 (x+y, y-x),$$
plotted in \cref{fig:beta-k2}, for which $\nabla \cdot \beta = 200$ and $\|\beta\|_{0,\,\infty} = 200$. This makes the (well-posed) problem strongly non-coercive with a medium high P\'eclet number. The latter example was also considered in \cite{Bur13} for numerical experiments on a non-coercive convection--diffusion equation with Cauchy data.

We consider the following domains for data assimilation, shown in \cref{fig:domains},
\begin{equation}\label{ex1}
\omega = (0.2,0.45)\times (0.2,0.45),\quad
B =  (0.2,0.45)\times (0.55,0.8),
\end{equation}
\begin{equation}\label{ex2}
\omega = (0,0.125)\times (0.4,0.6) \cup (0.875,1)\times (0.4,0.6),\quad
B =  (0.25,0.75)\times (0.4,0.6),
\end{equation}
\begin{equation}\label{ex3}
\omega = \Omega\setminus [0,0.875]\times [0.125,0.875],\quad
B =  \Omega\setminus [0,0.125]\times [0.125,0.875].
\end{equation}

\begin{figure}[h]
	\begin{subfigure}{0.32\textwidth}
		\includegraphics[draft=false, width=\textwidth]{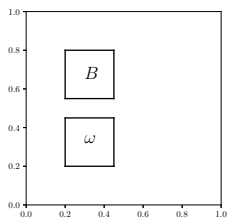}
		\caption{Boundaries for \eqref{ex1}.}
	\end{subfigure}
	\begin{subfigure}{0.32\textwidth}
		\includegraphics[draft=false, width=\textwidth]{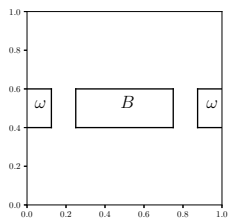}
		\caption{Boundaries for \eqref{ex2}.}
	\end{subfigure}
	\hfill
	\begin{subfigure}{0.32\textwidth}
		\includegraphics[draft=false, width=\textwidth]{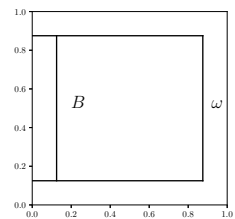}
		\caption{Boundaries for \eqref{ex3}.}
	\end{subfigure}
	\caption{Computational domains.}
	\label{fig:domains}
\end{figure}
 
\begin{figure}[h]
	\begin{subfigure}{0.48\textwidth}
		\includegraphics[width=\textwidth]{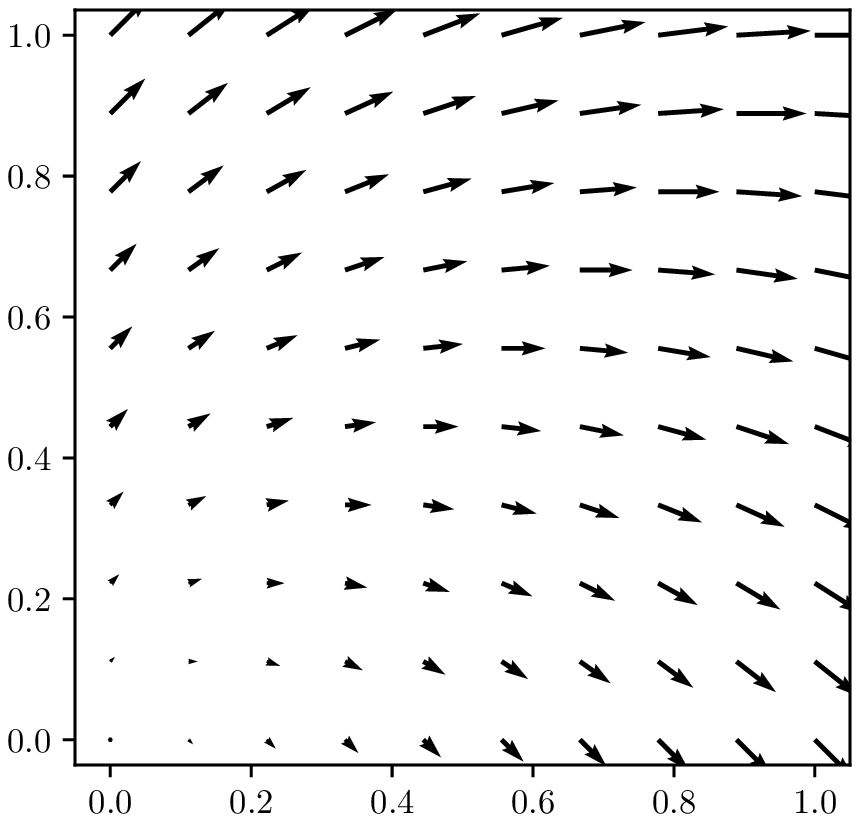}
	\end{subfigure}
	\begin{subfigure}{0.48\textwidth}
		\includegraphics[width=\textwidth]{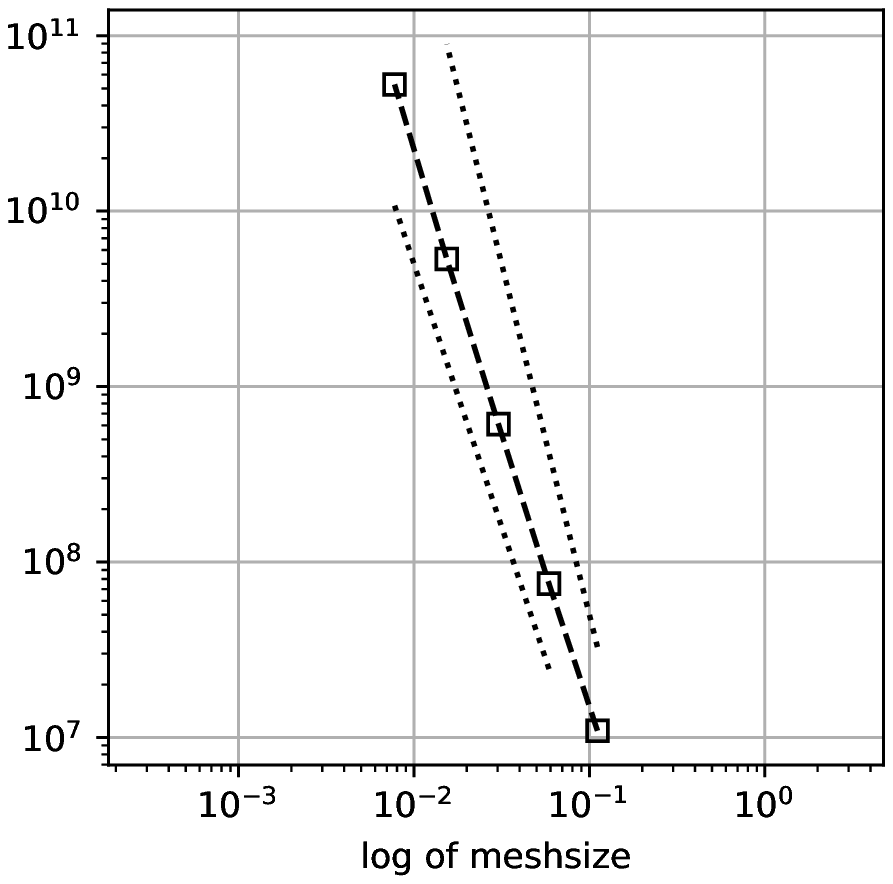}
	\end{subfigure}
	\caption{Left: convection field $\beta_{nc}$. Right: condition number $\mathcal{K}_2$ for domains \eqref{ex1}, $\beta=\beta_c$; the dotted lines are proportional to $h^{-3}$ and $h^{-4}$.}
	\label{fig:beta-k2}
\end{figure}

The condition number upper bound in \cref{prop:unique} is illustrated for a particular case in \cref{fig:beta-k2}, where we plot the condition number $\mathcal{K}_2$ versus the mesh size $h$, together with reference dotted lines proportional to $h^{-3}$ and $h^{-4}$. For five meshes with $2^N$ elements on each side, $N=3,\ldots,7$, the approximate rates for $\mathcal{K}_2$ are -3.03, -3.16, -3.2, -3.34.

\begin{figure}[h]
\begin{subfigure}{0.48\textwidth}
	\includegraphics[draft=false, width=\textwidth]{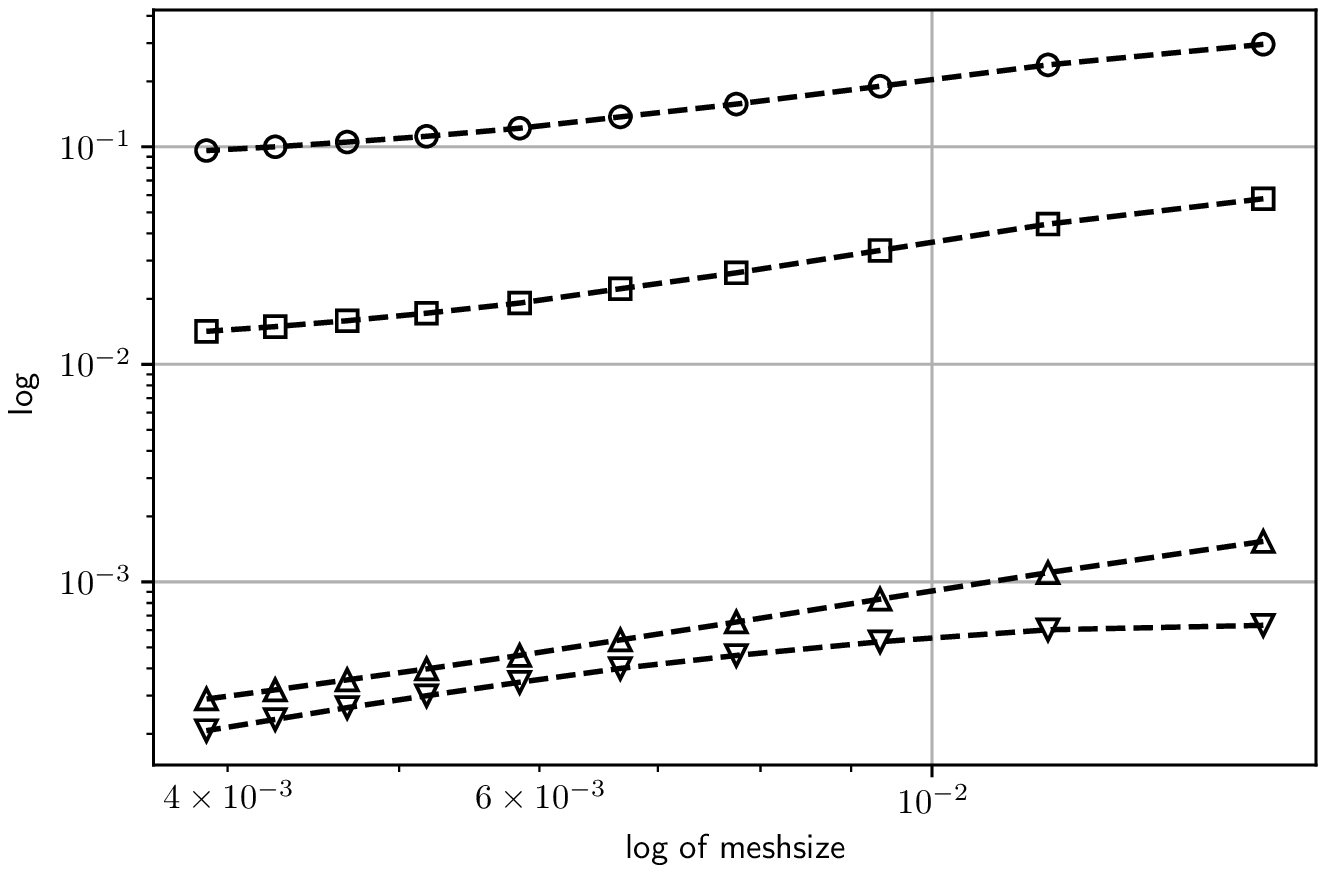}
	\caption{Circles: $H^1$-error, rate $\approx 0.45$; Squares: $L^2$-error, rate $\approx 0.56$; Up triangles: $s(e_h,e_h)^\frac12$, rate $\approx 1.1$; Down triangles: $s_*(z_h,z_h)^\frac12$, rate $\approx 1.33$.}
\end{subfigure}
\hfill
\begin{subfigure}{0.49\textwidth}
	\includegraphics[draft=false, width=\textwidth]{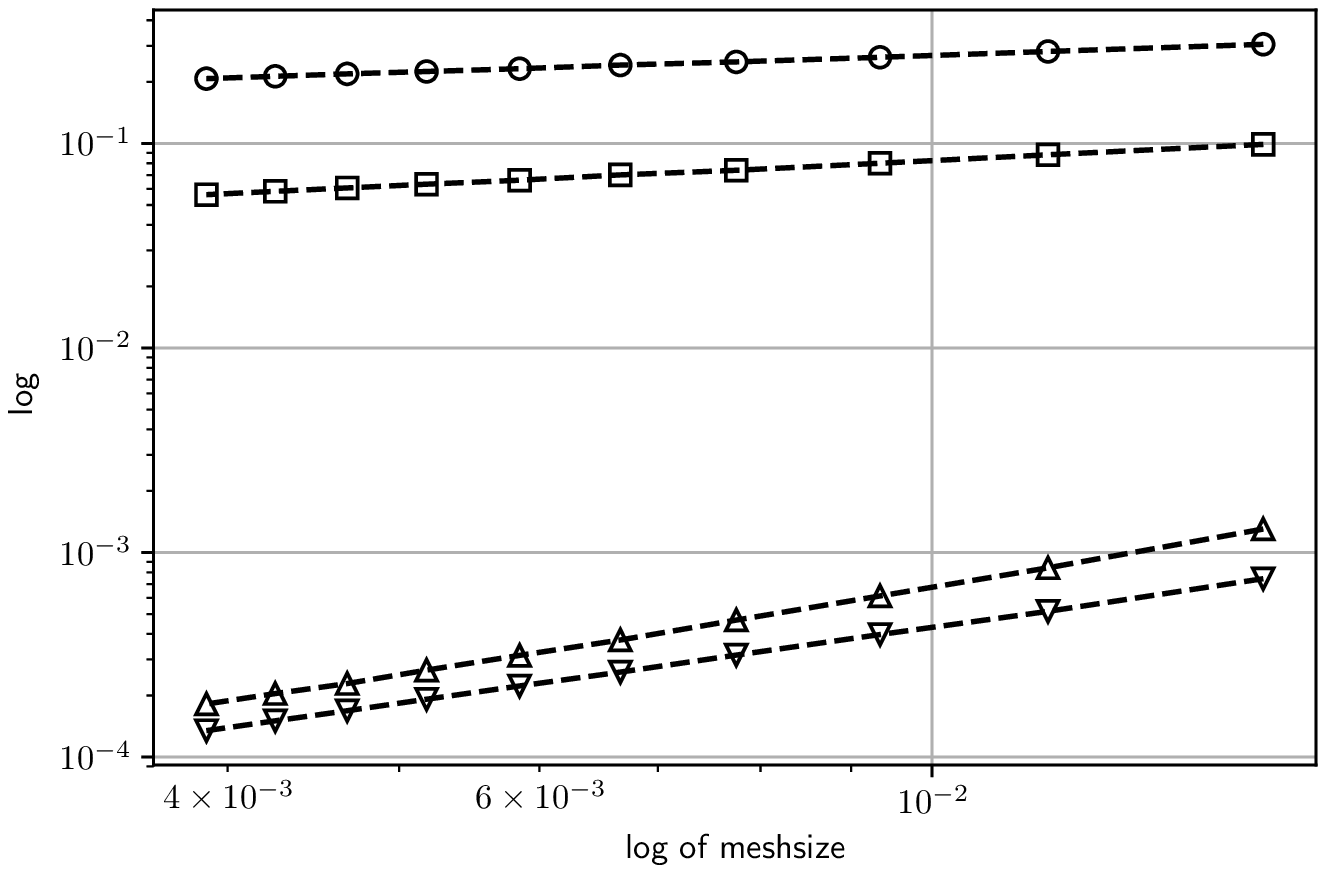}
	\caption{Circles: $H^1$-error, rate $\approx 0.29$; Squares: $L^2$-error, rate $\approx 0.42$; Up triangles: $s(e_h,e_h)^\frac12$, rate $\approx 1.32$; Down triangles: $s_*(z_h,z_h)^\frac12$, rate $\approx 1.34$.}
\end{subfigure}
\caption{Convergence for domains \eqref{ex1}. Left: $\beta = \beta_c$. Right: $\beta = \beta_{nc}$.}
\label{fig:conv1}
\end{figure}

The results in \cref{fig:conv1} for the domains \eqref{ex1} strongly agree with the convergence rates expected from \cref{L^2-error} and \cref{H^1-error} for the relative errors in $B$ computed in the $L^2$- and $H^1$-norms, and with the rates for $\|(e_h,z_h)\|_s$ given in \cref{prop:conv_stab}. 

The numerical approximation improves when considering the setting in \eqref{ex2}, in which data is given both downstream and upstream, as reported in \cref{fig:conv2}.
The convergence is almost linear and the size of the errors is considerably reduced in the non-coercive case.

The resolution increases all the more when data is given near a big part of the boundary $\partial \Omega$, as for the computational domains \eqref{ex3} considered in \cref{fig:conv3}.
In this configuration of the set $\omega$, for both convective fields $\beta_{c}$ and $\beta_{nc}$, the $L^2$-errors decrease below $10^{-4}$ with superlinear rates on the same meshes considered in \cref{fig:conv1} and \cref{fig:conv2}.

\begin{figure}[h]
	\begin{subfigure}{0.48\textwidth}
		\includegraphics[draft=false, width=\textwidth]{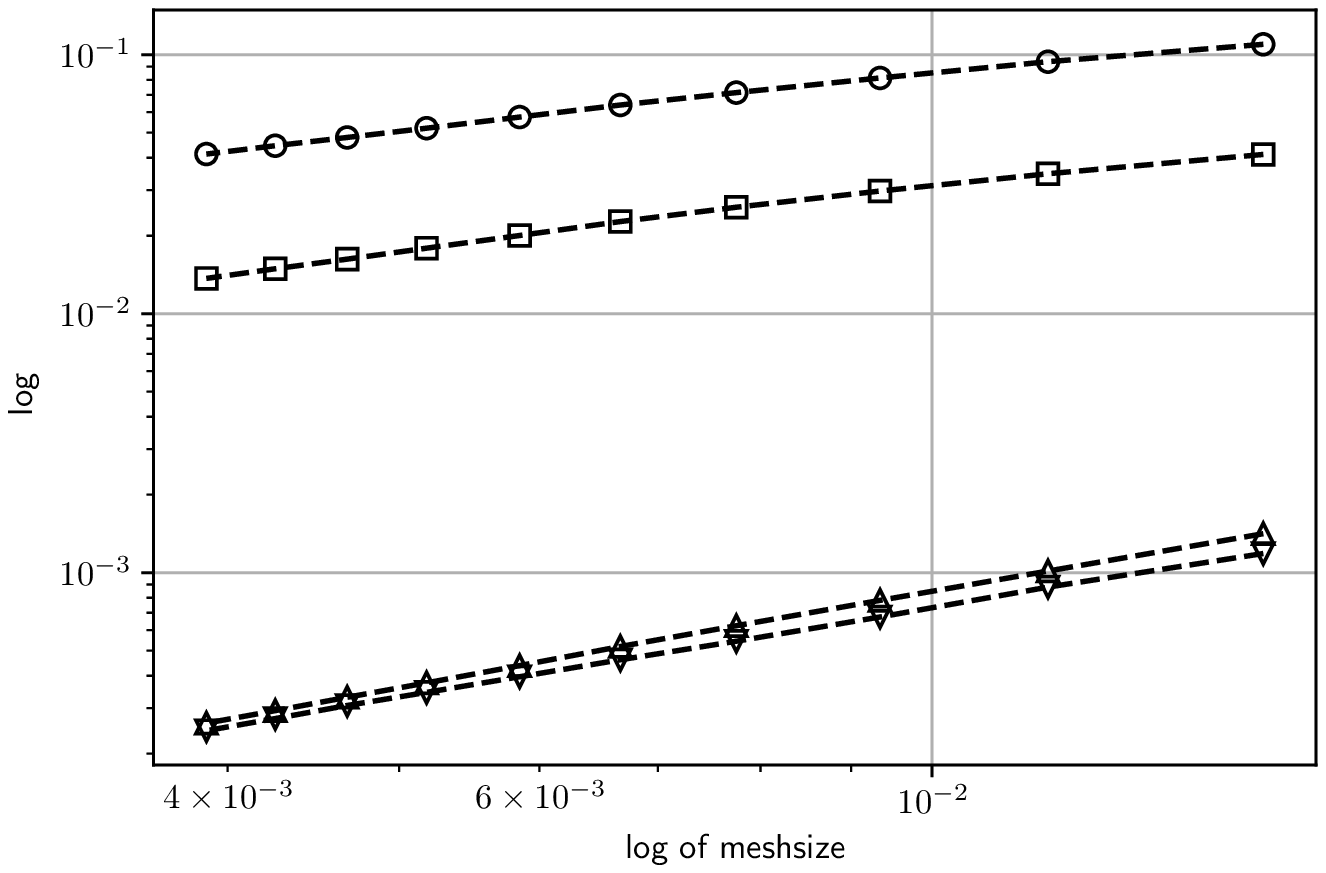}
		\caption{Circles: $H^1$-error, rate $\approx 0.8$; Squares: $L^2$-error, rate $\approx 0.94$; Up triangles: $s(e_h,e_h)^\frac12$, rate $\approx 1.24$; Down triangles: $s_*(z_h,z_h)^\frac12$, rate $\approx 1.2$.}
	\end{subfigure}
	\hfill
	\begin{subfigure}{0.48\textwidth}
		\includegraphics[draft=false, width=\textwidth]{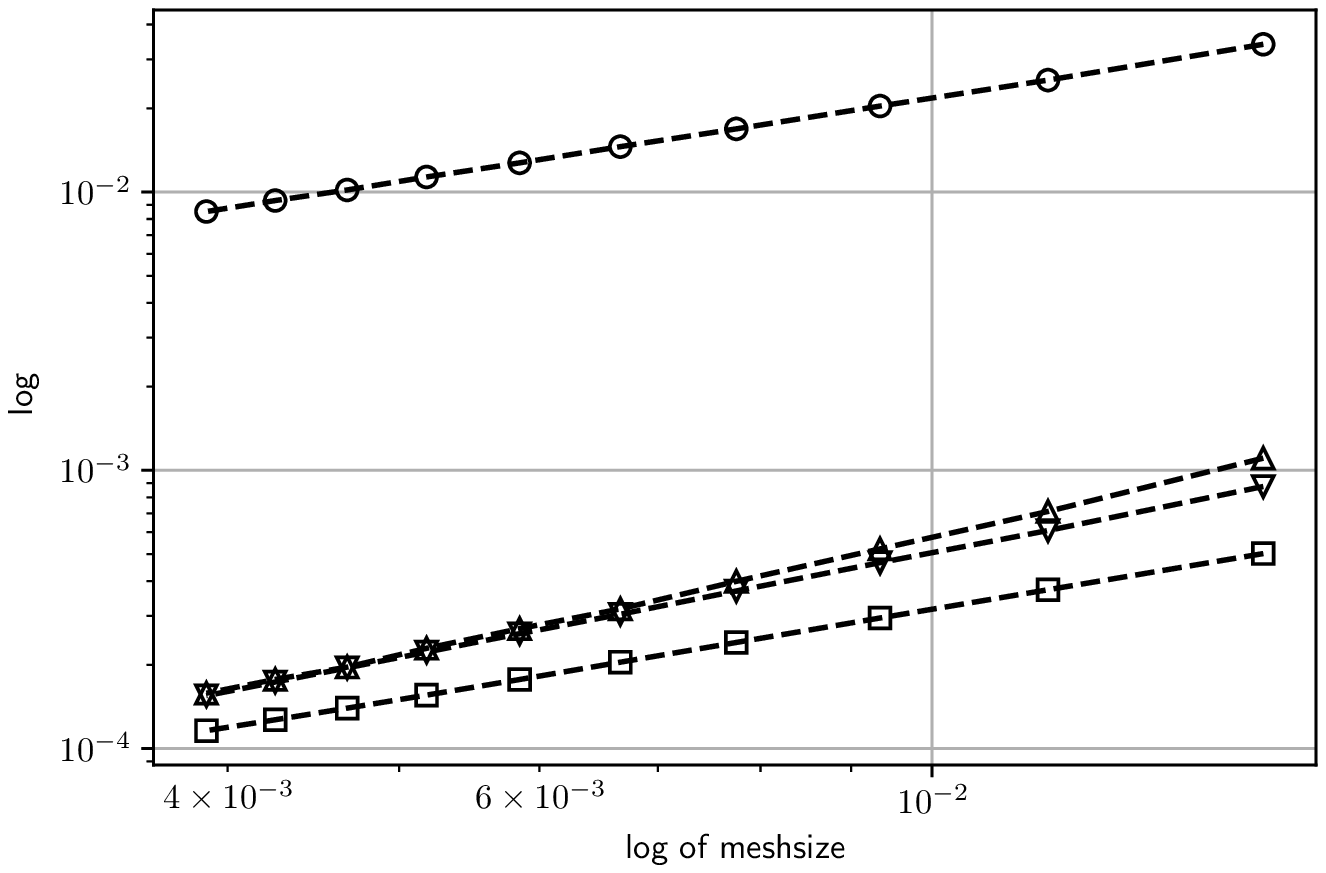}
		\caption{Circles: $H^1$-error, rate $\approx 1.02$; Squares: $L^2$-error, rate $\approx 1.07$; Up triangles: $s(e_h,e_h)^\frac12$, rate $\approx 1.3$; Down triangles: $s_*(z_h,z_h)^\frac12$, rate $\approx 1.25$.}
	\end{subfigure}
	\caption{Convergence for domains \eqref{ex2}. Left: $\beta = \beta_c$. Right: $\beta = \beta_{nc}$.}
	\label{fig:conv2}
\end{figure}

\begin{figure}[h]
	\begin{subfigure}{0.48\textwidth}
		\includegraphics[draft=false, width=\textwidth]{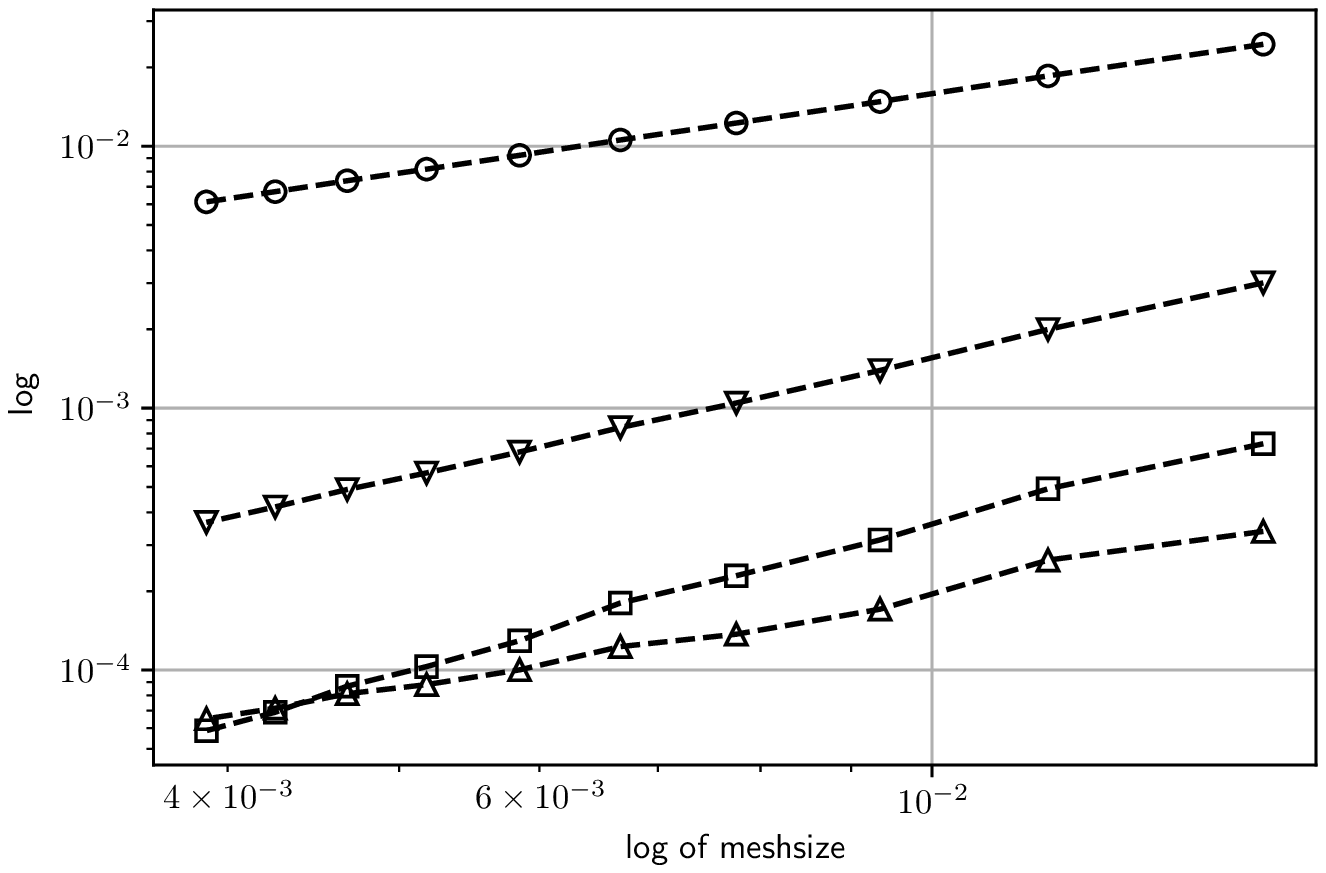}
		\caption{Circles: $H^1$-error, rate $\approx 1$; Squares: $L^2$-error, rate $\approx 1.81$; Up triangles: $s(e_h,e_h)^\frac12$, rate $\approx 1.04$; Down triangles: $s_*(z_h,z_h)^\frac12$, rate $\approx 1.52$.}
	\end{subfigure}
	\hfill
	\begin{subfigure}{0.48\textwidth}
		\includegraphics[draft=false, width=\textwidth]{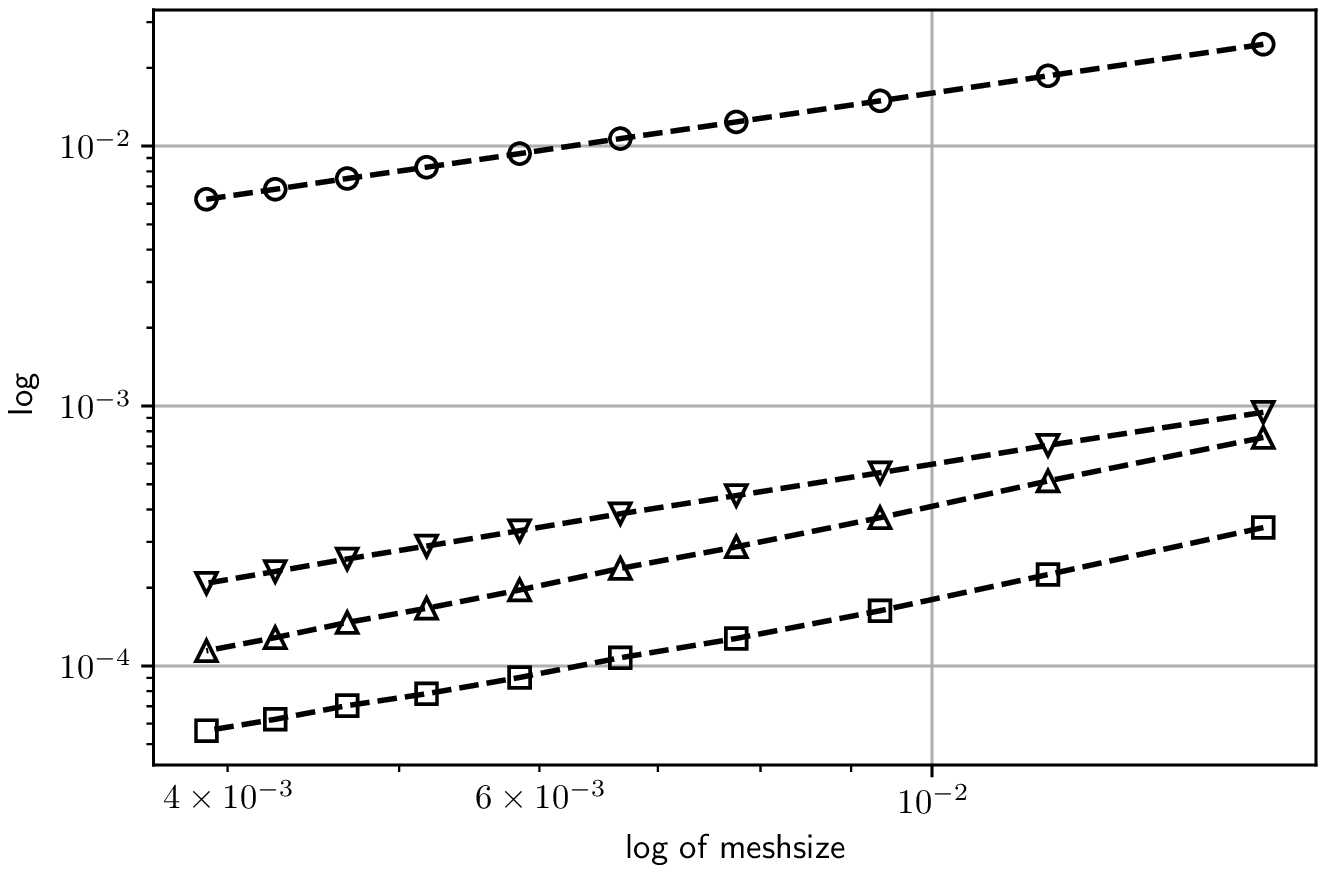}
		\caption{Circles: $H^1$-error, rate $\approx 1$; Squares: $L^2$-error, rate $\approx 1.13$; Up triangles: $s(e_h,e_h)^\frac12$, rate $\approx 1.30$; Down triangles: $s_*(z_h,z_h)^\frac12$, rate $\approx 1.16$.}
	\end{subfigure}
	\caption{Convergence for domains \eqref{ex3}. Left: $\beta = \beta_c$. Right: $\beta = \beta_{nc}$.}
	\label{fig:conv3}
\end{figure}

Comparing the geometries in \eqref{ex1} and \eqref{ex2} we also expect to see different effects of the two convective fields $\beta_{c}$ and $\beta_{nc}$. Notice that for both geometries the horizontal magnitude of $\beta_{nc}$ is greater than that of $\beta_{c}$. In \eqref{ex1} the solution is continued in the crosswind direction for both $\beta_{c}$ and $\beta_{nc}$, and a stronger convective field is not expected to improve the reconstruction. On the other side, in \eqref{ex2} information is propagated both downstream and upstream, and a stronger convective field can improve the resolution, despite the increase in the P\'eclet number. Indeed, we can see in \cref{fig:conv1} that for the geometry in \eqref{ex1} the numerical approximation is better for $\beta_c$ than for $\beta_{nc}$, while \cref{fig:conv2} shows better results for $\beta_{nc}$ than for $\beta_{c}$ in the case of \eqref{ex2}, especially for the $L^2$-error. 

\begin{figure}
	\begin{subfigure}{0.48\textwidth}
		\includegraphics[draft=false, width=\textwidth]{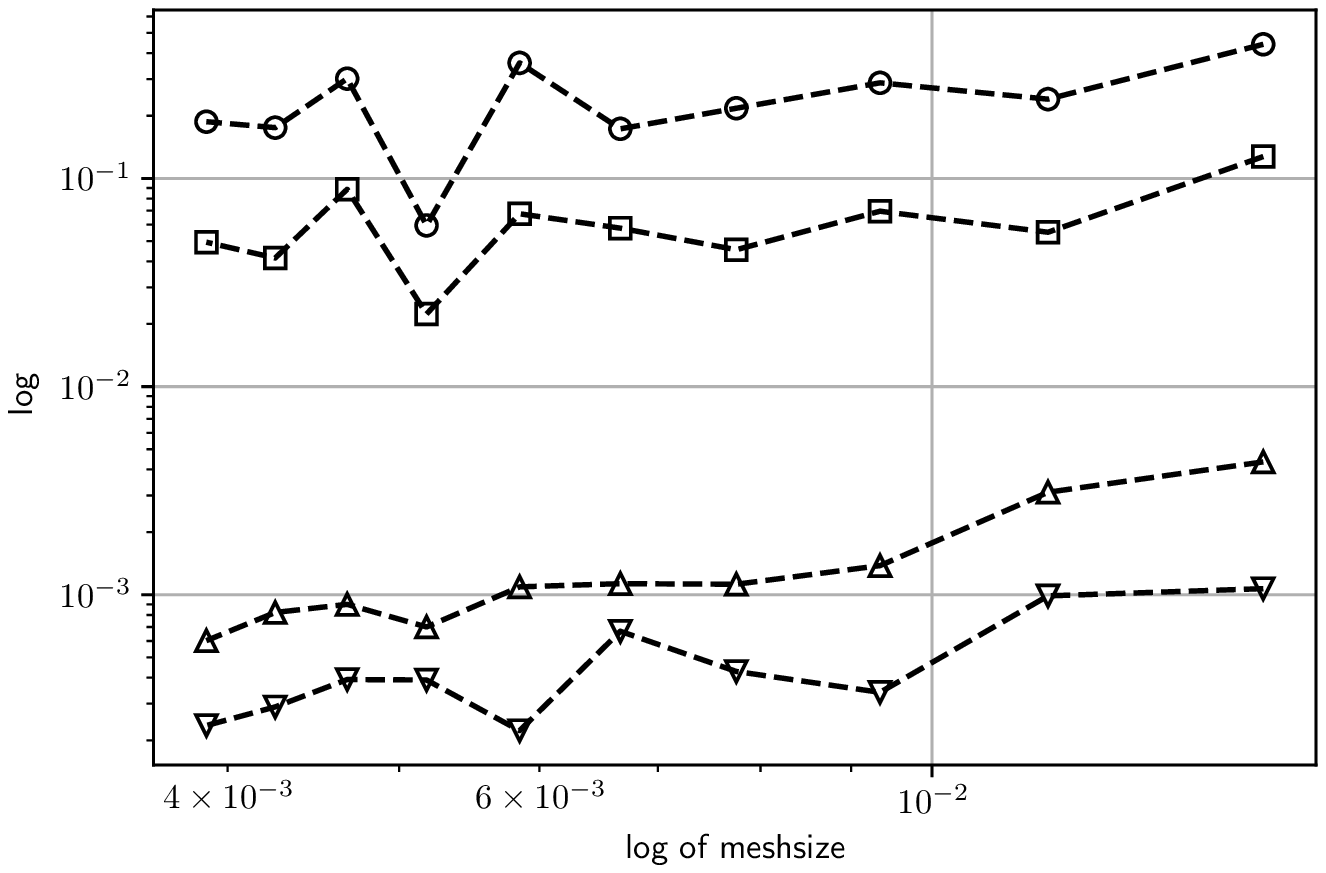}
		\caption{Noise amplitude $O(h^\frac12)$.}
	\end{subfigure}
	\hfill
	\begin{subfigure}{0.48\textwidth}
		\includegraphics[draft=false, width=\textwidth]{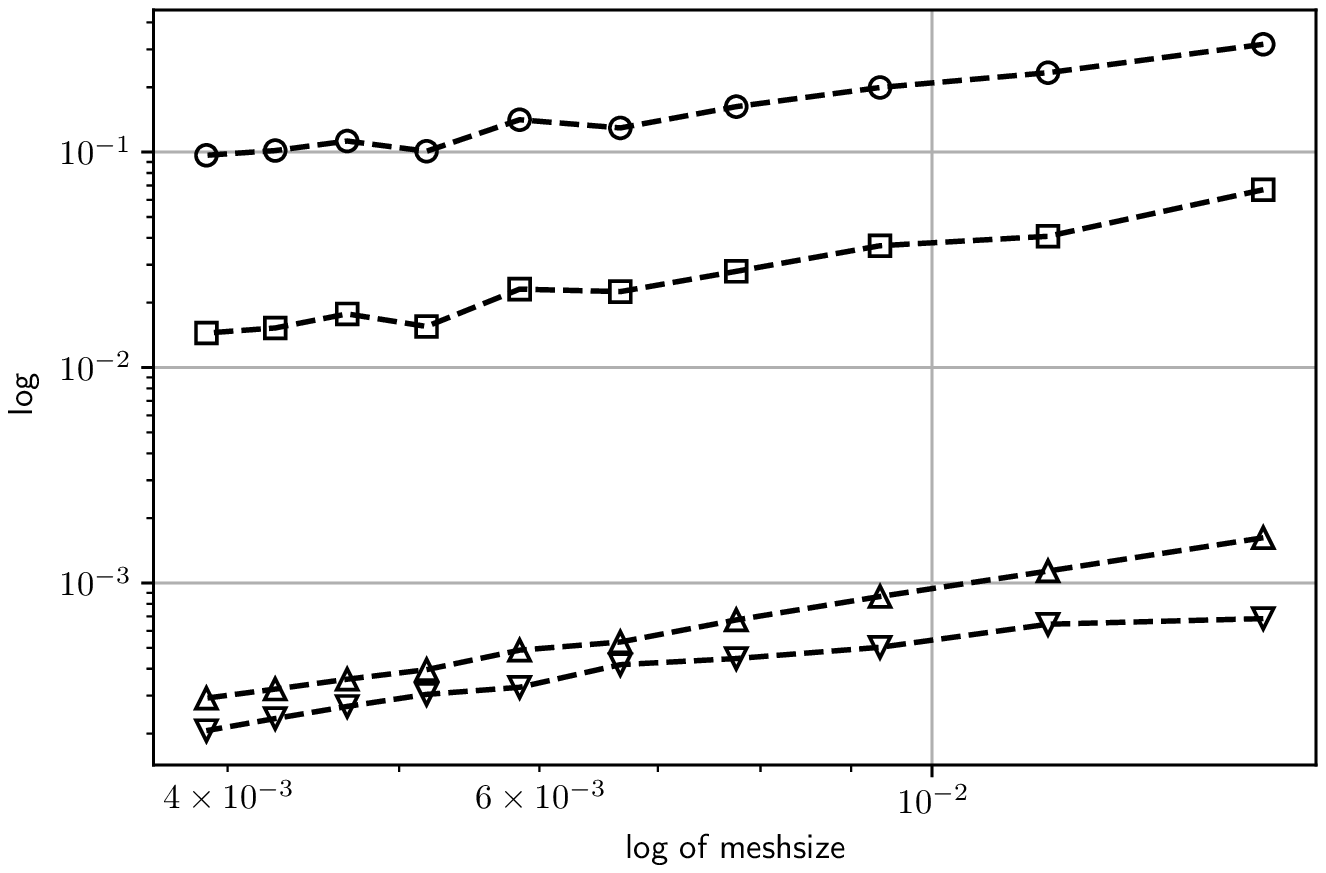}
		\caption{Noise amplitude $O(h)$.}
	\end{subfigure}
	\caption{Convergence for perturbed $\tilde U_\omega$ in domains \eqref{ex1}, $\beta = \beta_{c}$.}
	\label{fig:perturbation_c}
\end{figure}

To exemplify the noisy data $\tilde{U}_\omega = u\vert_{\omega} + \delta$, we perturb the restriction of $u$ to $\omega$ on every node of the mesh with uniformly distributed values in $[-h^\frac12, h^\frac12]$, respectively $[-h,h]$. Recall that by the error estimates in \cref{sec:FEM} the contribution of the perturbation $\delta$ is bounded by $h^{-1} \|\delta\|_\omega$. It can be seen in \cref{fig:perturbation_c} that the perturbations are strongly visible for an $O(h^\frac12)$ amplitude, but not for an $O(h)$ one.

\appendix
\section{}\label{appendix}
Denote by $(\cdot, \cdot)$, $|\cdot|$, $\div$, $\nabla$ and $D^2$ the inner product, norm,  divergence, gradient and Hessian in the Euclidean setting of $\Omega \subset \R^n$. We recall the following identity \cite[Lemma 1]{BNO19}.

\begin{lemma}
	\label{lem:carleman_eq}
	Let $\ell, w \in C^2(\Omega)$ and $\sigma \in C^1(\Omega)$.
	We define $v = e^\ell w$ and 
	\begin{equation*}
	a = \sigma - \Delta \ell, \quad 
	q = a + |\nabla \ell|^2, \quad 
	b = -\sigma v - 2(\grad v, \grad \ell), \quad
	B = (|\grad v|^2 - q v^2) \grad \ell.
	\end{equation*}
	Then
	\begin{align*}
	e^{2 \ell} (\Delta w)^2/2 
	&=
	(\Delta v + q v)^2/2 + b^2/2
	\\&	+ a |\nabla v|^2 + 2 D^2 \ell(\grad v, \grad v)
	+ \left(-a |\grad \ell|^2 + 2 D^2 \ell (\grad \ell, \grad \ell)\right)v^2
	\\&	+ \div(b \nabla v + B) + R,
	\end{align*}
	where 
	$R = (\grad \sigma , \grad v)v + \left(\div (a \grad \ell) - a\sigma \right) v^2.$
\end{lemma}

\begin{proof}[Proof of \cref{prop:carleman}]
	Let $\ell = \tau \phi$ and let $\lambda>0$ such that $|D^2 \rho(X,X)| \le \lambda |X|^2$. Recalling that $\phi = e^{\alpha \rho}$ and using the product rule we have that
	$$
	D^2\phi(X,X) = \alpha \phi(\alpha (\nabla \rho, X)^2 + D^2 \rho(X,X)),
	$$
	hence
	$$
	D^2\phi(X,X) \ge \alpha \phi D^2 \rho(X,X) \ge - \alpha \lambda \phi |X|^2.
	$$
	Combining this with the previous equality, we obtain
	\begin{equation*}
	D^2 \phi(\nabla \phi, \nabla \phi)
	\ge \alpha^3 \phi^3(\alpha |\nabla \rho|^4 - \lambda |\nabla \rho|^2).
	\end{equation*}
	Choosing $\epsilon>0$ such that $\epsilon \le |\nabla \rho|^2 \le \epsilon^{-1}$ it holds
	$$
	D^2 \phi(\nabla \phi, \nabla \phi) \ge \alpha^3 \phi^3(\alpha \epsilon^2 - \lambda \epsilon^{-1}).
	$$
	Since
	$$
	2 D^2 \ell(\nabla v, \nabla v) \ge -2 \alpha \lambda \phi \tau |\nabla v|^2,
	$$
	by choosing $\sigma = \Delta \ell + 3 \alpha \lambda \phi \tau$, i.e. $a = 3 \alpha \lambda \phi \tau$ in \cref{lem:carleman_eq} we obtain the bounds
	\begin{align*}
	a |\nabla v|^2 + 2 D^2 \ell(\grad v, \grad v) &\ge \alpha \lambda \phi \tau |\nabla v|^2,\\
	(-a |\grad \ell|^2 + 2 D^2 \ell (\grad \ell, \grad \ell))v^2 &\ge (2 \alpha \epsilon ^2 - (3 + 2\lambda) \epsilon^{-1}) (\alpha \phi \tau)^3 v^2.
	\end{align*}
	We now bound
	$$
	(\grad \sigma , \grad v)v = (\grad (\Delta \ell), \grad v)v + 3\alpha \lambda (\grad \ell , \grad v)v
	\ge - \left( |\grad (\Delta \phi)|  + 3\alpha \lambda |\grad \phi| \right) \tau |\grad v| |v|
	$$
	and
	$$
	(\div(a\grad \ell) - a\sigma)v^2 = ((\grad a, \grad \ell) - a^2)v^2 \ge (3\alpha \lambda |\grad \phi|^2- 9 \alpha^2 \lambda^2 \phi^2) \tau^2 v^2.
	$$
	Combining these lower bounds with
	\begin{align*}
	\tau |\grad v| |v|
	\le \frac{1}{2} (|\grad v|^2 + \tau^2 |v|^2),
	\end{align*}
	and taking $\alpha$ large enough, we obtain from \cref{lem:carleman_eq} that
	\begin{equation}\label{eq:Carleman-explicit}
	C e^{2 \tau\phi} (\Delta w)^2 \ge (a_1 \tau^3 - a_2 \tau^2)v^2 + (b_1 \tau - b_0)|\grad v|^2 + \div(b \nabla v + B),
	\end{equation} 
	with $a_j,b_j > 0$ depending only on $\alpha$, $\phi$ and $\lambda$. Taking $\tau$ large enough and using the elementary inequality
	\begin{align*}
	|\grad v|^2 = e^{2\tau \phi} |\tau w \grad \phi + \grad w|^2
	\ge e^{2\tau \phi} \frac{1}{2} |\grad w|^2 - e^{2\tau \phi}|\grad \phi|^2 \tau^2 w^2,
	\end{align*}
	we conclude by integrating over $K$ and using the divergence theorem.
\end{proof}

\section{}\label{appendixB}
We briefly recall herein the definition of semiclassical pseudodifferential operators and semiclassical Sobolev spaces. We then discuss the composition rule of two such operators, which is also called symbol calculus, and some estimates that are used in the proof of \cref{lem:shifted3b}. This presentation is based on \cite[Chapter 4]{Zwo12} and \cite[Section 2]{LeRL12}, to which we refer the reader for more details.

We shall use the following standard notation. For $\xi\in \R^n$ we set $\langle \xi \rangle = (1+|\xi|^2)^{\frac12}$, and for a multi-index $\alpha=(\alpha_1,\ldots,\alpha_n)\in \N^n$ let $|\alpha|=\alpha_1+\ldots \alpha_n$, $\alpha! = \alpha_1! \cdots \alpha_n!$, $\xi^\alpha = \xi_1^{\alpha_1}\cdots \xi_n^{\alpha_n}$. Let also $\partial^\alpha = \partial_{x_1}^{\alpha_1}\cdots \partial_{x_n}^{\alpha_n}$, $D=\frac{1}{i} \partial$ and $D^\alpha = \frac{1}{i^{|\alpha|}} \partial^\alpha$. The Schwartz space $\mathcal{S}(\R^n)$ is the set of rapidly decreasing $C^\infty$ functions and its dual $\mathcal{S}'(\R^n)$ is the set of tempered distributions. The semiclassical parameter $\h$ is assumed to be small: $\h\in(0,\h_0)$ with $\h_0 \ll 1$.

The semiclassical Fourier transform is a rescaled version of the standard Fourier transform. It is given by
$$
\mathcal{F}_\h \varphi (\xi) := \int_{\R^n} e^{-\frac{i}{\h}x\cdot \xi} \varphi(x) \mathrm{d}x 
$$
and its inverse is
$$
\mathcal{F}^{-1}_\h \psi (x) := \frac{1}{(2\pi\h)^n} \int_{\R^n} e^{\frac{i}{\h}x\cdot \xi} \psi(\xi) \mathrm{d}\xi. 
$$
The following properties hold: $\mathcal{F}_\h((\h D_x)^\alpha \varphi) = \xi^\alpha \mathcal{F}_\h \varphi$ and $(\h D_\xi)^\alpha \mathcal{F}_\h \varphi (\xi) = \mathcal{F}_\h ((-x)^\alpha \varphi)$.

\subsection{Symbol classes}
For $m\in \R$ the symbol class $S^m$ consists of functions $a(x,\xi,\h) \in C^\infty(\R^n \times \R^n)$ such that for all multi-indices $\alpha,\tilde{\alpha}\in \N^n$ there exists a constant $C_{\alpha,\tilde{\alpha}}>0$ uniform in $\h\in(0,\h_0)$ such that
$$|\partial_x^\alpha \partial_\xi^{\tilde{\alpha}} a(x,\xi,\h)| \le C_{\alpha,\tilde{\alpha}} \langle \xi \rangle^{m-|\tilde{\alpha}|}, \quad x\in \R^n,\, \xi\in \R^n.$$
Symbols in $S^m$ thus behave roughly as polynomials of degree $m$. We write that $a\in \h^N S^{m}$ if
$$|\partial_x^\alpha \partial_\xi^{\tilde{\alpha}} a(x,\xi,\h)| \le C_{\alpha,\tilde{\alpha}} \h^N \langle \xi \rangle^{m-|\tilde{\alpha}|}, \quad x\in \R^n,\, \xi\in \R^n.$$

\begin{lemma-non}[Asymptotic series]
	Let $m\in \R$ and the symbols $a_j \in S^{m-j}$ for $j=0,1,\ldots$. Then there exists a symbol $a\in S^m$ such that $a\sim \sum\limits_{j=0}^{\infty} \h^j a_j$, that is for every $N\in \N$,
	$$a-\sum_{j=0}^{N} \h^j a_j \in \h^{N+1} S^{m-N-1}.$$
	The symbol $a$ is unique up to $\h^\infty S^{-\infty}$, in the sense that the difference of two such symbols is in $\h^NS^{-M}$ for all $N,M\in \R$. The principal symbol of $a$ is given by $a_0$.
\end{lemma-non}
\subsection{Pseudodifferential operators}
Using these symbol classes we can define semiclassical pseudodifferential operators ($\psi$DOs).
For a symbol $a\in S^m$ we define the corresponding semiclassical $\psi$DO of order $m$, $Op(a):\mathcal{S}(\R^n)\to \mathcal{S}(\R^n)$,
$$Op(a)u(x) := \frac{1}{(2\pi\h)^n} \int_{\R^n} \int_{\R^n} e^{\frac{i}{\h}(x-y)\cdot \xi} a(x,\xi, \h)u(y) \mathrm{d}y \mathrm{d}\xi.$$
This is also called quantization of the symbol. $Op(a)$ can be extended to $\mathcal{S}'(\R^n)$ and $Op(a):\mathcal{S}'(\R^n)\to \mathcal{S}'(\R^n)$ continuously.
Note that $Op(a)u(x) = \mathcal{F}^{-1}_\h (a(x,\cdot)\mathcal{F}_\h u(\cdot))$ and that the operator corresponding to the symbol $a(x,\xi) = \sum_{|\alpha|\le N} a_\alpha(x) \xi^\alpha$ is $Op(a)u = \sum_{|\alpha|\le N} a_\alpha(x) (\h D)^\alpha u$. Notice that each derivative of this operator scales with $\h$.

For the present paper the most important example is the second order differential operator $A = -\h^2 \Delta + \h^2 \sum_{j=1}^{n} \beta_j(x) \partial_j$. Its symbol is given by $a(x,\xi,\h) = |\xi|^2 + i\h \sum_{j=1}^{n} \beta_j(x) \xi_j$, and its principal symbol is $a_0(x,\xi,\h) = |\xi|^2$.

\subsection{Semiclassical Sobolev spaces}\label{sec:sobolev}
For $s\in \R$ the semiclassical Sobolev spaces $H_\scl^s(\R^n)$ are algebraically equal to the standard Sobolev spaces $H^s(\R^n)$ but are endowed with different norms
$$
\norm{u}_{H_\scl^s(\R^n)} = \norm{J^s u}_{L^2(\R^n)},
$$
where the semiclassical Bessel potential is defined by $J^s = Op(\langle \xi \rangle^s)$. Informally,
$$
J^s = (1-\h^2 \Delta)^{s/2}, \quad s \in \R.
$$
For example, $\norm{u}_{H_\scl^1(\R^n)}^2 = \norm{u}_{L^2(\R^n)}^2 + \norm{\h \nabla u}_{L^2(\R^n)}^2$.
A semiclassical $\psi$DO of order $m$ is continuous from $H_\scl^{s}(\R^n)$ to $H_\scl^{s-m}(\R^n)$.

\subsection{Composition}
Composition of semiclassical $\psi$DOs can be analysed using the following calculus.
\begin{theorem-non}[Symbol calculus]\label{thm:composition}
	Let $a\in S^m$ and $b\in S^{m'}$. Then $Op(a)\circ Op(b) = Op(a\#b)$ for a certain $a\#b\in S^{m+m'}$ that admits the following asymptotic series
	\begin{equation*}
	a\#b(x,\xi,\h) \sim \sum_{\alpha} \frac{\h^{|\alpha|} i^{|\alpha|}}{\alpha!} D^\alpha_\xi a(x,\xi,\h) D^\alpha_x b(x,\xi,\h).
	\end{equation*}
\end{theorem-non}
The commutator and disjoint support estimates \eqref{commutator} and \eqref{pseudolocality} follow, respectively, from the following.
\begin{corollary-non}[Commutator and disjoint support]
	Let $a\in S^m$ and $b\in S^{m'}$. Then
	\begin{enumerate}[(i)]
		\item $a\#b - b\#a \in \h S^{m+m'-1}.$
		\item If $\supp(a)\cap \supp(b) = \emptyset$, then $a\#b\in \h^{\infty} S^{-\infty}$, i.e. $a\#b \in \h^NS^{-M}$ for all $N,M\in \R$.
	\end{enumerate}
\begin{proof}
	(i) The principal symbol of $a\#b$, that is the first term in its asymptotic series, is $ab$. The second term is $\frac{\h}{i} \sum_{j=1}^n \partial_{\xi_j}a(x,\xi,\h) \partial_{x_j}b(x,\xi,\h)$. We thus have that the principal symbol of the commutator $[Op(a),Op(b)] = Op(a\#b - b\#a)$ is given by
	\begin{equation*}
	\frac{\h}{i} \sum_{j=1}^n (\partial_{\xi_j}a \partial_{x_j}b - \partial_{x_j}a \partial_{\xi_j}b) \in \h S^{m+m'-1}. 
	\end{equation*}
	(ii) If $\supp(a)\cap \supp(b) = \emptyset$, then each term in the asymptotic series of $a\#b$ vanishes.
\end{proof}
\end{corollary-non}


\section*{Acknowledgements}
E.B. was supported by EPSRC grants EP/P01576X/1 and EP/P012434/1, and L.O. by EPSRC grants EP/P01593X/1 and EP/R002207/1.

\bibliographystyle{alpha} 
\bibliography{biblio_part1}

\end{document}